\documentclass[12pt, reqno]{article}
\usepackage{amsmath}
\usepackage[margin=1.2in]{geometry}
\usepackage{amssymb,amscd,graphicx}
\usepackage{braket}
\usepackage{dsfont}
\usepackage{soul, color, xcolor}
\usepackage{authblk}
\usepackage{times}
\usepackage{cite}
\usepackage{amsthm}
\usepackage{algorithm}
\usepackage{algorithmicx}
\usepackage{algpseudocode}
\usepackage{subfigure}
\usepackage{multirow}
\usepackage{float}
\usepackage{titling}
\usepackage{tikz}
\usepackage{enumitem}
\usepackage{bm} 
\usepackage{graphicx}
\usepackage{epstopdf}

\thanksmarkseries{arabic}
\usepackage{bbm}

\usepackage{caption}

\newtheorem{theorem}{Theorem}[section]

\newtheorem{lemma}{Lemma}[section]

\newtheorem{remark}{Remark}[section]

\numberwithin{equation}{section}

\sethlcolor{yellow}

\title{A Posteriori Error Estimation Improved by a Reconstruction Operator for the Stokes Optimal Control Problem}

\date{}

\author{Jingshi Li$^1$,  Jiachuan Zhang$^{2,*}$\\
$^1$School of Mathematics and Statistics, \\
Nanjing University of Information Science and Technology,\\
 Nanjing, Jiangsu, 210044, P. R. China \\ 
$^2$School of Physical and Mathematical Sciences, \\
Nanjing Tech University, \\
Nanjing, Jiangsu, 211816, P. R. China. \\
$^*$Corresponding author: Jiachuan Zhang (zhangjc@njtech.edu.cn)
}

\begin{document}
\maketitle
\begin{abstract}
This paper focuses on \textit{a posteriori} error estimates for a pressure-robust finite element method, which incorporates a divergence-free reconstruction operator,  within the context of  the distributed optimal control problem constrained by the Stokes equations. We develop an enhanced residual-based \textit{a posteriori} error estimator that is independent of pressure and establish its global reliability and efficiency.
The proposed \textit{a posteriori} error estimator enables the separation of velocity and pressure errors in \textit{a posteriori} error estimation, ensuring velocity-related estimates are free of pressure influence. Numerical experiments confirm our conclusions.


\end{abstract}

\section{Introduction}

The optimal control problems of the Stokes equations play a pivotal role in fluid dynamics and engineering applications, and it has been a hot topic to explore the error estimates for this problem. While \textit{a priori} error estimates for finite element solutions to Stokes optimal control problems have been extensively studied and are well-established , e.g. \cite{gunzburger1991analysis, gunzburger1991, deckelnick2004semidiscretization, bochev2004least, rosch2006optimal}, the realm of \textit{a posteriori} error estimates is still in its developmental stages and not yet fully understood. For the earliest work on \textit{a posteriori} error estimates for the Stokes optimal control problem, one can refer to the research by Liu and Yan \cite{Liu2002}. They give an upper bound in the \textit{a posteriori} error estimates for the finite element approximation of the control variable, velocity, and its adjoint variable. 
In recent years, Allendes et al. conducted two studies on this problem: the first addresses optimal control problems for the generalized Oseen system (encompassing the Stokes equations as a specific case), while the second focuses on point-wise tracking optimal control problems of the Stokes equations. Both works provide  \textit{a posteriori} error estimates for the finite element approximation with upper and lower bounds \cite{allendesPosterioriErrorEstimation2018a,allendes2019adaptive}.
Additionally, Yan et al. developed \textit{a posteriori} error bounds for nonconforming finite element solutions of Stokes optimal control problem \cite{shen2022convergence}.

However, existing \textit{a posteriori} error estimates in Stokes optimal control problem cannot decouple the velocity and pressure errors, such as \cite{Liu2002,allendes2019adaptive,shen2022convergence}, which can lead to a series of issues \cite{olshanskii2004grad, olshanskii2009grad, linke2014role, john2017divergence}: First, the velocity error becomes dependent on the pressure multiplied by the reciprocal of viscosity parameter. In detail, when the viscosity parameter is very small or the pressure is very large, the velocity error may become excessively large, a phenomenon particularly pronounced in fluid flow simulations. Moreover, this lack of robustness can also lead to numerical solution instability, which can affect the reliability and effectiveness of the simulation results. In practical applications, this could result in inaccurate predictions of flow characteristics, such as the distribution of velocity and pressure fields. Therefore, developing pressure-robust methods for velocity error estimation is crucial to ensure numerical stability and accuracy.

Currently, there are roughly three approaches that can achieve pressure robustness. The first method is based on the classic divergence-free mixed Galerkin scheme, which, through the new concept of finite element exterior calculus, has achieved consistency and divergence-free characteristics of the discrete velocity in the $H^1$ space \cite{buffa2011isogeometric}. The second method utilizes the discontinuous Galerkin technique, which, through a mixed scheme, seeks velocities that are H(div) consistent and divergence-free, thereby ensuring a certain level of mathematical regularity \cite{cockburn2007note, wang2009robust, kanschat2014divergence}.
The third method, employs a divergence-free reconstruction operator on the test functions of the right-hand side of the equations \cite{linke2012divergence, linke2014role, linke2016robust}. This technique effectively maps discretely divergence-free functions to their exact counterparts, thereby ensuring the pressure-robustness. In this paper, we primarily adopt the third method. Compared to the first two methods, the approach based on divergence-free reconstruction operators can achieve pressure robustness at a relatively low cost.

For the  incompressible Stokes equations $-\nu\triangle \boldsymbol{y}+\nabla p = \boldsymbol{f}, \nabla \cdot \boldsymbol{y}= 0$,  Linke first establish pressure-robust \textit{a priori} error estimates based on a reconstruction operator \cite{linke2014role}. Subsequently, Lederer et al. further advanced this work by developing pressure-robust \textit{a posteriori} error estimates \cite{ledererRefinedPosterioriError2019}.  They constructed \textit{a posteriori} error estimator that are independent of pressure, which was achieved by replacing the traditional  residual-based volume contribution $\|h_T (\boldsymbol{f}-\nabla p_h +\nu div_h(\nabla \boldsymbol{y}_h))\|_{L^2(T)}$ with $\|h_T^2 curl(\boldsymbol{f}+\nu div_h(\nabla \boldsymbol{y}_h))\|_{L^2(T)}$. Very recently, Merdon et al. studied the pressure-robust discretization of the Stokes optimal control problem based on a reconstruction operator, and concluded that the \textit{a priori} errors of the velocity and its adjoint are all independent of the pressure \cite{merdon2023pressure}. However, for the pressure-robust \textit{a posteriori} error estimates of the Stokes optimal control problem, to the best of our knowledge, no relevant conclusions can be found yet. 

In this paper, we focus on the pressure-robust \textit{a posteriori} error estimation in the distributed optimal control problem of the Stokes equations. By employing a divergence-free reconstruction operator, we obtain a pressure-robust finite element discretization of the original problem. The main contribution of this work is that we construct a pressure-independent \textit{a posteriori} error estimator for control variable, velocity, and the adjoint variables. The error estimator comprises three components: the discretization of the necessary optimality conditions,  the discretized state term with pressure removed, and the discretized adjoint term with the pressure's adjoint removed. We have also demonstrated the global efficiency and reliability of the error estimator.

This paper is organized as follows. Section 2 introduces the Stokes optimal control problem that we considered and some necessary notations.
Section 3 presents the classical finite element discretization of the Stokes optimal control problem and the standard residual-based \textit{a posteriori} error estimator. The divergence-free reconstruction operator is introduced in section 4, and we prove the \textit{a posteriori} error estimate bounds for pressure-robust finite element formulation of the control problem in this section. In section 5,  three numerical examples are conducted to verify our theoretical results. The numerical results indicate that our theory is also applicable to non-homogeneous boundary conditions.

\section{ Model problem and notations}
In this section, we present the model problem and establish the relevant notation and preliminary conclusions necessary for our later analysis.

\subsection{ The optimal control problem constrained by the Stokes equations}

Let $\Omega$ be a bounded, open, polygonal domain in $\mathds{R}^d$, where $d=2$ or $3$, with a Lipschitz boundary $\partial \Omega$.
Define the Sobolev spaces $H_0^1(\Omega):=\{y\in H^1(\Omega)|tr(u)=0~ on~ \partial \Omega \}$, in which $tr$ denotes the trace operator. And we introduce the spaces $\boldsymbol{V}=[H_0^1(\Omega)]^d$, $\boldsymbol{U}=[L^2(\Omega)]^d$, and $Q=L_0^2(\Omega):=\{q\in L^2(\Omega)|\int_{\Omega}qdx=0\}$. For any subset $G\subseteq \Omega$, we adopt the notations $\|\cdot\|_{m,G}$ and $|\cdot|_{m,G}$ to represent the Sobolev space norm and seminorm, respectively, associated with $W^{m,2}(G)$, and extend these notations to vector-valued functions. Specifically, when $G$ coincides with the domain $\Omega$ or when $m=0$, we may omit $\Omega$ or $m$ from the subscript of the norms for simplicity.

Given a desired state $\boldsymbol{y_d}\in \boldsymbol{V}$, we formulate the optimal control problem constrained by Stokes equations as follows: Find the state $\boldsymbol{y}\in \boldsymbol{V}$, the pressure $p\in Q$, and the control $\boldsymbol{u}\in \boldsymbol{U}$ that minimize the cost functional
\begin{align}
\label{Eq: cost function}
J(\boldsymbol{y},\boldsymbol{u}):=\frac{1}{2}\|\boldsymbol{y}-\boldsymbol{y_d}\|^2+\frac{\alpha}{2}\|\boldsymbol{u}\|^2
\end{align}
subject to the incompressible Stokes equations
\begin{align}
\label{Eq: Stokes}
\left\{\begin{aligned}
-\nu\triangle \boldsymbol{y}+\nabla p &= \boldsymbol{u}, ~&in ~~ \Omega,\\
\nabla \cdot \boldsymbol{y} &= 0, ~&in ~~ \Omega,\\
\boldsymbol{y}&=\boldsymbol{0}, ~ &on ~~\partial \Omega.
\end{aligned}
\right.
\end{align}
Here, $\nu>0$ represents the viscosity coefficient, $\boldsymbol{y_d}$ is the desired state vector, 
$\alpha>0$ is a regularization parameter.

To apply the finite element method, we recast the control problem given by equations (\ref{Eq: cost function})-(\ref{Eq: Stokes}) 
in the following weak formula: Find $(\boldsymbol{y}, p,\boldsymbol{u}) \in \boldsymbol{V}\times Q\times \boldsymbol{U}$ satisfying
\begin{align}\label{Eq: cost function weak}
\min J(\boldsymbol{y},\boldsymbol{u})
\end{align}
and the following variational equations
\begin{align}\label{Eq: Stokes weak}
\left\{\begin{aligned}
a(\boldsymbol{y},\boldsymbol{v})-b(\boldsymbol{v},p) &= (\boldsymbol{u},\boldsymbol{v}), ~&\forall \boldsymbol{v}\in \boldsymbol{V},\\
b(\boldsymbol{y},\phi) &= 0, ~&\forall  \phi \in Q,
\end{aligned}
\right.
\end{align}
where the bilinear forms are defined as
\begin{align*}
a(\boldsymbol{\varphi},\boldsymbol{\psi})&:=\int_\Omega \nu\nabla \boldsymbol{\varphi} \cdot \nabla \boldsymbol{\psi}, \quad\forall ~ (\boldsymbol{\varphi}, \boldsymbol{\psi})\in \boldsymbol{V}\times\boldsymbol{V},\\
 b(\boldsymbol{\varphi},\psi)&:=\int_\Omega  \psi \nabla \cdot \boldsymbol{\varphi}, \quad\forall ~  (\boldsymbol{\varphi}, \psi)\in \boldsymbol{V}\times Q,\\
 (\boldsymbol{\varphi},\boldsymbol{\psi})&:=\int_\Omega \boldsymbol{\varphi} \cdot \boldsymbol{\psi}, \quad\forall ~  (\boldsymbol{\varphi}, \boldsymbol{\psi})\in \boldsymbol{U}\times\boldsymbol{U}.
\end{align*}

For the state equation, we have the following \textit{a priori} estimate of the velocity and pressure.
\begin{lemma}(see Theorem~5.4 in \cite{girault2012finite})\label{Lem: stability estimate}
If $\textit{\textbf{f}}\in \boldsymbol{U}$, and $(\textbf{w},\rho)$ is the solution of the following equation
\begin{align*}
a(\textbf{w},\boldsymbol{v})\pm b(\boldsymbol{v},\rho)=(
\textit{\textbf{f}},\boldsymbol{v}),~~~~\forall~ \boldsymbol{v}\in \boldsymbol{V},\\
b(\textbf{w},\phi)=0,~~~~\forall~ \phi\in Q.
\end{align*}
Then, it has
\begin{align*}
\|\textbf{w}\|_1+\|\rho\|_0\leq C\|\textit{\textbf{f}}\|_0,
\end{align*}
where $C$ is independent of $\textbf{w}, \rho, \textit{\textbf{f}}$.
\end{lemma}

\begin{remark}
The original property of Theorem~5.4 in \cite{girault2012finite} is:
\begin{align*}
\|\textbf{w}\|_1+\|\rho\|_0\leq C\|\textit{\textbf{f}}\|_{-1}.
\end{align*}
From Cauchy-Schwarz inequality and Poincar\'e inequality with constant $C_p$, we have
\begin{align*}
\|\boldsymbol{f}\|_{-1}=\sup_{\boldsymbol{v}\in \boldsymbol{V}}\frac{(\boldsymbol{f},\boldsymbol{v})}{\|\nabla\boldsymbol{v}\|_0}\leq \sup_{\boldsymbol{v}\in \boldsymbol{V}}\frac{\|\boldsymbol{f}\|_0\|\boldsymbol{v}\|_0}{\|\nabla\boldsymbol{v}\|_0}
\leq \sup_{\boldsymbol{v}\in \boldsymbol{V}}\frac{C_p\|\boldsymbol{f}\|_0\|\nabla\boldsymbol{v}\|_0}{\|\nabla\boldsymbol{v}\|_0}
=C_p\|\boldsymbol{f}\|_0,
\end{align*}
which implies the conclusion in Lemma~\ref{Lem: stability estimate}.
\end{remark}

Incorporating the adjoint velocity $\boldsymbol{z}\in \boldsymbol{V}$ and the adjoint pressure $r\in Q$, we obtain the dual form of the state equation (\ref{Eq: Stokes weak}): 
\begin{align*}
\left\{\begin{aligned}
a(\boldsymbol{v},\boldsymbol{z})+b(\boldsymbol{v},r) &= (\boldsymbol{y}-\boldsymbol{y_d},\boldsymbol{v}), ~\forall~ \boldsymbol{v}\in \boldsymbol{V},\\
b(\boldsymbol{z},\phi) &= 0, ~\forall~  \phi \in Q.
\end{aligned}
\right.
\end{align*}

We now present the existence and uniqueness of the optimal solutions through the following lemma, which is fundamental to our analysis. 
\begin{lemma}(see \cite{Edmunds1972})\label{Thm: opt sys}
The optimal control problem (\ref{Eq: cost function weak})-(\ref{Eq: Stokes weak}) has a unique optimal solution $(\bar{\boldsymbol{y}},\bar{p},\bar{\boldsymbol{u}})\in \boldsymbol{V}\times Q \times \boldsymbol{U}$, which together with the optimal adjoint variables $(\bar{\boldsymbol{z}},\bar{r})\in \boldsymbol{V}\times Q$ satisfies the following optimality system
\begin{align}\label{Eq: opt sys}
\begin{aligned}
a(\bar{\boldsymbol{y}},\boldsymbol{v})-b(\boldsymbol{v},\bar{p})&= (\bar{\boldsymbol{u}},\boldsymbol{v}), &~\forall~ \boldsymbol{v}\in \boldsymbol{V},\\
b(\bar{\boldsymbol{y}},\phi)&= 0, &~\forall  \phi \in Q,\\
a(\boldsymbol{w},\bar{\boldsymbol{z}})+b(\boldsymbol{w},\bar{r})&= (\bar{\boldsymbol{y}}-\boldsymbol{y_d},\boldsymbol{w}), &~\forall~ \boldsymbol{w}\in \boldsymbol{V},\\
b(\bar{\boldsymbol{z}},\psi)&= 0, &~\forall~  \psi \in Q,\\
(\bar{\boldsymbol{z}}+\alpha
\bar{\boldsymbol{u}},\boldsymbol{\mu})&= 0,&~ \forall~
\boldsymbol{\mu}\in \boldsymbol{U}.
\end{aligned}
\end{align}
The final equation of the system is called the first-order necessary condition for optimality, which can be directly written as
\begin{align}\label{Eq: 1 order necessary optimal condition}
 \bar{\boldsymbol{u}}= -\frac{1}{\alpha}\bar{\boldsymbol{z}}.
\end{align}
\end{lemma}

\subsection{Notation}

Denote $\mathcal{T}_h=\bigcup_{i=1}^N T_i$ as a shape regular (conforming) partition of $\bar{\Omega}$. 
Assume $\mathcal{E}_h$ to be the set of edges within this partition. And the set of all the internal edges is denoted by $\mathcal{E}_h\setminus \partial \Omega$.
For each element $T\in \mathcal{T}_h$ and edge $E\in \mathcal{E}_h$, let $h_T$ and $h_E$ represent their respective diameters.
Define $\mathcal{T}_E$ as the collection of all triangular elements containing the edge $E$. 
The global mesh size $h$ is given by  $h:=max\{h_T:T\in \mathcal{T}_h\}$. 
Let $\boldsymbol{n}_E$ and $ \boldsymbol{\tau}_E$ 
denote the unit normal vector and the unit tangential vector of $E\in \mathcal{E}_h$, respectively. For internal edges $E\in \mathcal{E}_h\setminus \partial \Omega$, the direction of $\boldsymbol{n}_E$ is arbitrary but fixed, whereas for boundary edges $E\in \partial \Omega$,
$\boldsymbol{n}_E$ is oriented outward. 

We define the jump of a piecewise continuous function $v_h$ across an internal edge $E\in \mathcal{E}_h\setminus \partial \Omega$ as
\begin{align*}
  [v_h]:=v_h|_{T_1}-v_h|_{T_2},
\end{align*}
where $T_1$ and $T_2$ are the elements sharing the edge $E$, with $T_1\cup T_2=\mathcal{T}_E$.

Furthermore, $\mathcal{P}_k(T)$ represents the space of polynomials of degree at most $k$ on $T\in \mathcal{T}_h$. 
And the space $\mathcal{P}_k^+(T) $ is an extension of $\mathcal{P}_k(T)$, 
enriched with  standard cell bubble functions (for definition, see reference \cite{crouzeix1973conforming}) for $k\geq2$.

Suppose  ${\boldsymbol I}_k$ is the local interpolation 
operator that interpolates into the space $[\mathcal{P}_k(T)]^d$.
For $k\geq 0$, it has the following local interpolation estimate for ${\boldsymbol I}_k$ (cf. \cite{ern2004theory})
\begin{align}\label{Eq: interpolation estimate}
  \|\boldsymbol{v}-{\boldsymbol I}_k\boldsymbol{v}\|_{0,T}\lesssim h_T\|\boldsymbol{v}\|_{1,T},\quad \forall T\in \mathcal{T}_h, \forall \boldsymbol{v}\in \boldsymbol{V}.
\end{align}

\begin{remark}
In this work we use $\mathbbm{a}\lesssim \mathbbm{b}$ when there exists a constant $c$ independent of $\mathbbm{a},\mathbbm{b}, h$ such that $\mathbbm{a}\leq c\mathbbm{b}$.
\end{remark}

Let $M$ represent either a triangular element $T$ or an edge $E$, with $h_M$ being the diameter of $M$. And $\mathcal{M}$ is the collection of all such elements or edges.
The data oscillation is then defined as
\begin{equation}\label{Def: osc}
  osc(\cdot,\mathcal{M},k)^2:=\sum_{M\in \mathcal{M}} h_M^2\|(1-\Pi_k)\cdot\|_M^2,
\end{equation}
where $\Pi_k$ is the locally $L^2$-projection operator into $\mathcal{P}_k(T)$ or $[\mathcal{P}_k(T)]^d$ for $T\in  \mathcal{T}_h$.
For the specific case when $k=2$, we denote the oscillation simply as $osc(\cdot,\mathcal{M})$.

\section{Classical finite element discretization and standard \textit{a posteriori} error estimates}
When employing classical finite element methods to discretize the Stokes equations, the standard residual-based \textit{a posteriori} error estimator for velocity typically incorporates a volume contribution that is associated with the pressure term. Under conditions of low viscosity, this can lead to the locking phenomenon. Similar difficulties are also encountered when considering Stokes optimal control problems. This section displays the classical finite element formulation for solving the optimal control problem of Stokes equations and provide the standard \textit{a posteriori} error estimation based on residuals.

Define the conforming  finite element spaces as 
\begin{align*}
  \boldsymbol{V}_h:=\{\boldsymbol{v}_h\in \boldsymbol{V}:\boldsymbol{v}_h|_T\in [\mathcal{P}_2^+(T)]^d , ~\forall T\in \mathcal{T}_h  \},
\end{align*}
\begin{align*}
Q_h:=\{q_h\in Q:q_h|_T\in \mathcal{P}_1(T), ~\forall T\in \mathcal{T}_h  \}.
\end{align*}
The above finite element spaces for velocity and pressure satisfy the inf-sup condition (see \cite{linke2016robust}): there exists a constant $\gamma$ such that 
\begin{align*}
  \inf_{q_h\in Q_h\setminus \{0\}} \sup_{\boldsymbol{v}_h\in \boldsymbol{V}_h\setminus \{0\}} \frac{\int_{\Omega}q_h \nabla\cdot \boldsymbol{v}_h dx}{\|q_h\|\|\nabla \boldsymbol{v}_h\|}=\gamma>0.
\end{align*}

Define
\begin{align*}
\boldsymbol{U}_h:=\{\boldsymbol{u}_h\in \boldsymbol{U}:\boldsymbol{u}_h|_T\in [\mathcal{P}_2(T)]^d, ~\forall T\in \mathcal{T}_h  \}.
\end{align*}

With these spaces, we give the finite element approximation for the optimal control problem (\ref{Eq: cost function weak})-(\ref{Eq: Stokes weak}) as: Find the optimal  $(\boldsymbol{y}_h,p_h,\boldsymbol{u}_h)\in \boldsymbol{V}_h\times Q_h\times \boldsymbol{U}_h$ such that
\begin{align} 
&\min J_h(\boldsymbol{y}_h,\boldsymbol{u}_h):=\frac{1}{2}\|\boldsymbol{y}_h-\boldsymbol{y_d}\|^2+\frac{\alpha}{2}\|\boldsymbol{u}_h\|^2
 \label{Eq: cost function classical finite}\\
s.t.&\left\{\begin{aligned}\label{Eq: Stokes classical finite}
a(\boldsymbol{y}_h,\boldsymbol{v}_h)-b(\boldsymbol{v}_h,p_h) = (\boldsymbol{u}_h,\boldsymbol{v}_h), ~\forall \boldsymbol{v}_h\in \boldsymbol{V}_h,\\
b(\boldsymbol{y}_h,\phi_h) = 0, ~\forall  \phi_h \in Q_h.
\end{aligned}
\right.
\end{align}

It is well known that the discrete optimal control problem (\ref{Eq: cost function classical finite})-(\ref{Eq: Stokes classical finite}) has a unique solution $(\bar{\boldsymbol{y}}_h^c,\bar{p}_h^c,\bar{\boldsymbol{u}}_h^c)\in \boldsymbol{V}_h\times Q_h\times \boldsymbol{U}_h$. And the optimal solution together with the optimal adjoint variables
$(\bar{\boldsymbol{z}}_h^c,\bar{r}_h^c)\in \boldsymbol{V}_h\times Q_h$ satisfy  the following discrete optimality system
\begin{align}\label{Eq: opt sys classical finite}
\begin{aligned}
a(\bar{\boldsymbol{y}}_h^c,\boldsymbol{v}_h)-b(\boldsymbol{v}_h,\bar{p}_h^c)&= (\bar{\boldsymbol{u}}_h^c,\boldsymbol{v}_h), &~\forall \boldsymbol{v}_h\in \boldsymbol{V}_h,\\
b(\bar{\boldsymbol{y}}_h^c,\phi_h)&= 0, &~\forall  \phi_h \in Q_h,\\
a(\boldsymbol{w}_h,\bar{\boldsymbol{z}}_h^c)+b(\boldsymbol{w}_h,\bar{r}_h^c)&= (\bar{\boldsymbol{y}}_h^c-\boldsymbol{y_d},\boldsymbol{w}_h), &~\forall \boldsymbol{w}_h\in \boldsymbol{V}_h,\\
b(\bar{\boldsymbol{z}}_h^c,\psi_h)&= 0, &~\forall  \psi_h \in Q_h,\\
(\bar{\boldsymbol{z}}_h^c+\alpha
\bar{\boldsymbol{u}}_h^c,\boldsymbol{\mu}_h)&= 0,&~ \forall
\boldsymbol{\mu}_h\in \boldsymbol{U}_h.
\end{aligned}
\end{align}

To provide a comparative perspective with the pressure-robust \textit{a posteriori} error estimation that we will introduce later, the following theorem presents the classical residual-based \textit{a posteriori} error estimation for the control, velocity, and its corresponding adjoint variable.

\begin{theorem}\label{thm3}
Suppose $(\bar{\boldsymbol{y}}, \bar{p},\bar{\boldsymbol{z}},\bar{r},\bar{\boldsymbol{u}})$ is  the solution of the optimality system (\ref{Eq: opt sys}).
Then the following estimates hold

\begin{equation}\label{Eq: Stokes reliability classical}
\|\bar{\boldsymbol{u}}-\bar{\boldsymbol{u}}_h^c\|+\|\nabla(\bar{\boldsymbol{y}}-\bar{\boldsymbol{y}}_h^c)\|+\|\nabla(\bar{\boldsymbol{z}}-\bar{\boldsymbol{z}}_h^c)\| \lesssim \hat{\varepsilon},
\end{equation}
and
\begin{equation}\label{Eq: Stokes efficiency classical}
\hat{\varepsilon}\lesssim\|\bar{\boldsymbol{u}}-\bar{\boldsymbol{u}}_h^c\|+\|\nabla(\bar{\boldsymbol{y}}-\bar{\boldsymbol{y}}_h^c)\|+\|\nabla(\bar{\boldsymbol{z}}-\bar{\boldsymbol{z}}_h^c)\|+\|\bar{p}-\bar{p}_h^c\|+\|\bar{r}-\bar{r}_h^c\|+OSC_{1},
\end{equation}
with the definition of  $\hat{\varepsilon}$ given by
 \begin{align}
 \begin{aligned}\label{Def: error estimator classical}
   \hat{\varepsilon}^2 :=& \frac{1}{\nu^2}\sum_{T\in \mathcal{T}_h}\| h_T (\bar{\boldsymbol{u}}_h^c-\nabla \bar{p}_h^c+\nu \Delta\bar{\boldsymbol{y}}_h^c)\|_{T}^2+\sum_{E\in \mathcal{E}_h\setminus \partial \Omega}\|h_E^{1/2}[ \nabla \bar{\boldsymbol{y}}_h^c\cdot
  \boldsymbol{n}_E]\|_{E}^2\\
  & +\frac{1}{\nu^2}\sum_{T\in \mathcal{T}_h}\| h_T (\bar{\boldsymbol{y}}_h^c-\boldsymbol{y}_d+\nabla \bar{r}_h^c+\nu \Delta\bar{\boldsymbol{z}}_h^c)\|_{T}^2+\sum_{E\in \mathcal{E}_h\setminus \partial \Omega}\|h_E^{1/2}[ \nabla \bar{\boldsymbol{z}}_h^c\cdot
  \boldsymbol{n}_E]\|_{E}^2\\
  & +\|\nabla\cdot\bar{\boldsymbol{y}}_h^c\|^2+\|\nabla\cdot\bar{\boldsymbol{z}}_h^c\|^2+\|\frac{\bar{\boldsymbol{z}}_h^c}{\alpha}+\bar{\boldsymbol{u}}_h^c\|^2.
    \end{aligned}
 \end{align}
 And
 \begin{align*}
   OSC_{1}=1/\nu (osc(\bar{\boldsymbol{u}}_h^c-\nabla \bar{p}_h^c,\mathcal{T}_h)+osc(\bar{\boldsymbol{y}}_h^c-\boldsymbol{y}_d-\nabla \bar{r}_h^c,\mathcal{T}_h)),
 \end{align*}
 where the definition of the data oscilation  $osc(\cdot,\mathcal{T}_h)$  can be found in the special case of  (\ref{Def: osc}).

\end{theorem}

\begin{remark}
The proof relies on  Theorem 3 and Theorem 4 in \cite{ledererRefinedPosterioriError2019}.
And the verification of (\ref{Eq: Stokes reliability classical}) and  (\ref{Eq: Stokes efficiency classical}) draw parallels with those of Theorems \ref{thm2} and \ref{thm4}. The detailed proof procedure will be presented in the subsequent section. The conclusion indicates that within the classical framework of \textit{a posteriori} error estimation, it is not possible to isolate the velocity error from the pressure error (see \cite{hannukainen2012unified,allendes2019adaptive}).
\end{remark}

\section{Pressure-robust finite element discretization and improved \textit{a posteriori} error estimates}

Recent studies have shown that adding divergence-free reconstruction operator to the right-hand side of the Stokes equations can resolve the aforementioned issues, resulting in \textit{a priori} and \textit{a posteriori} velocity error estimates that are independent of pressure. Pressure-robust \textit{a priori} error estimates  have been obtained for the Stokes optimal control problem (see\cite{merdon2023pressure}). This prompts us to investigate the pressure-robust \textit{a posteriori} error estimates for the optimal control problem. In this section, we will develop \textit{a posteriori} error estimator independent of pressure and demonstrate global upper and lower bounds for ( \ref{Eq: cost function finite})-(\ref{Eq: Stokes finite}). Before that, we first introduce the reconstruction operator and list the conclusion of the existing pressure-robust \textit{a posteriori} error estimates for the Stokes equations.

\subsection{Reconstruction operator and existing pressure-robust results for the Stokes equations}

Focus on the following weak form of the standard Stokes equations:  For $\boldsymbol{f}\in \boldsymbol{U}$, find $(\boldsymbol{y},p)\in \boldsymbol{V}\times Q$ such that
\begin{align}\label{Eq: Stokes weak 1}
\begin{aligned}
a(\boldsymbol{y},\boldsymbol{v})-b(\boldsymbol{v},p)&= (\boldsymbol{f},\boldsymbol{v}), &~\forall \boldsymbol{v}\in \boldsymbol{V},\\
b(\boldsymbol{y},\phi)&= 0, &~\forall  \phi \in Q.
\end{aligned}
\end{align}

Denote a discrete divergence-free function space by
\begin{align}
\boldsymbol{V}_h^0:=\{ \boldsymbol{v}_h\in \boldsymbol{V_h}: argmin_{q_h\in Q_h}\|div~\boldsymbol{v_h}-q_h\|=0\}.
\end{align}

The Stokes equations are recognized for their invariance property (see \cite{linke2014role}): a modification of the right-hand side force vector $\boldsymbol{f}$ to $\boldsymbol{f}+\nabla \phi$ induces a corresponding adjustment in the solution set, shifting from $(\boldsymbol{y},p)$ to $(\boldsymbol{y},p+\phi)$. This presents the pressure-robustness. However, the classical finite element discretization for (\ref{Eq: Stokes weak 1}) does not preserve the invariance property. An alternative approach is to construct a reconstruction operator to ensure that the mapped test functions on the right-hand side of the equation are divergence-free.

Set a discrete finite element space as
\begin{align*}
  \boldsymbol{W}_h:=\{ \boldsymbol{w}_h\in \boldsymbol{U}:\boldsymbol{w}_h|_T\in \boldsymbol{RT}_{1}(T), \forall T\in \mathcal{T}_h ,( \phi, \boldsymbol{w}_h\cdot \boldsymbol{n}_E)_E=0, \forall \phi\in \mathcal{P}_{1}(E), \forall E\in \mathcal{E}_h\},
\end{align*}
where $ \boldsymbol{RT}_{1}(T)$ is the Raviart–Thomas space defined in \cite{raviart1977primal}.

Suppose there is a reconstruction operator $\mathcal{R}_h: \boldsymbol{V}_h\rightarrow \boldsymbol{W}_h$ that satisfies the following conditions simultaneously,
\begin{align}
 &(i)\quad \nabla\cdot(\mathcal{R}_h \boldsymbol{v}_h)=0 , ~\text{and} ~(\mathcal{R}_h \boldsymbol{v}_h\cdot \boldsymbol{n}_E)|_{E\in\partial \Omega}=0,  ~\forall \boldsymbol{v}_h\in \boldsymbol{V}_h^0. \label{Eq: div} \\
& (ii)\quad \text{For}~\forall T\in\mathcal{T}_h, ~\text{if} ~\boldsymbol{v}_h\in \boldsymbol{V}_h ~\text{and}~ \boldsymbol{w}_h\in [\mathcal{P}_{0}(T)]^d, ~\text{then}~ (\boldsymbol{v}_h-\mathcal{R}_h \boldsymbol{v}_h,\boldsymbol{w}_h)_T = 0.  \label{Eq: reconstruction operator def2}\\
 &(iii)\quad \|\boldsymbol{v}_h-\mathcal{R}_h\boldsymbol{v}_h\|_{T}\lesssim h_T^m|\boldsymbol{v}_h|_{m,T},  ~ \forall \boldsymbol{v}_h\in \boldsymbol{V}_h, T\in\mathcal{T}_h, m=0,1,2.\label{Eq: 1-R estimate}
\end{align}

The existence of a reconstruction operator that complies with conditions (i)-(iii) is proven in \cite{linke2016robust}, where additional details are also elaborated.

With the reconstruction operator, the pressure-robust finite element approximation of  (\ref{Eq: Stokes weak 1}) is given by: Find $(\boldsymbol{y}_h, p_h)\in \boldsymbol{V}_h\times Q_h$ such that
\begin{align}\label{Eq: Stokes dis 1}
\begin{aligned}
a(\boldsymbol{y}_h,\boldsymbol{v}_h)-b(\boldsymbol{v}_h,p_h) &= (\boldsymbol{f},\mathcal{R}_h\boldsymbol{v}_h), ~&\forall \boldsymbol{v}_h\in \boldsymbol{V}_h,\\
b(\boldsymbol{y}_h,\phi_h) &= 0, ~&\forall  \phi_h \in Q_h.
\end{aligned}
\end{align}

Recall that the curl operator for the two-dimensional vector $\boldsymbol{v}=(v_1,v_2)\in [H^1(\Omega)]^2$  is defined as
$curl(\boldsymbol{v}):=\frac{\partial v_2}{\partial x}-\frac{\partial v_1}{\partial y}$.

For velocity $\boldsymbol{y}_h$ and forcing function $\boldsymbol{f}$ within equation (\ref{Eq: Stokes dis 1}), 
define a novel \textit{a posteriori} error estimator functional $\varepsilon: \boldsymbol{V}_h\times\boldsymbol{U}\rightarrow \mathbb{R}$ for the Stokes equations  which  is first introduced in \cite{Lederer2017}\textbf{} and defined as 
 \begin{align}
 \begin{aligned}\label{Def: error estimator}
   \varepsilon(\boldsymbol{y}_h,\boldsymbol{f})^2 :=& \frac{1}{\nu^2}\sum_{T\in \mathcal{T}_h}\| h_T^2
curl(\boldsymbol{f}+\nu \Delta\boldsymbol{y}_h)\|_{T}^2+\sum_{E\in \mathcal{E}_h\setminus \partial \Omega}\|h_E^{1/2}[ \nabla \boldsymbol{y}_h\cdot
  \boldsymbol{n}_E]\|_{E}^2\\
    &+\frac{1}{\nu^2}\sum_{E\in \mathcal{E}_h\setminus \partial \Omega}\|h_E^{3/2}[(\boldsymbol{f}+\nu  \Delta\boldsymbol{y}_h)\cdot \boldsymbol{\tau}_E]\|_{E}^2+\|\nabla\cdot\boldsymbol{y}_h\|^2.
    \end{aligned}
 \end{align}

The subsequent theorem provides the upper and lower bounds for the \textit{a posteriori} velocity error estimate.

\begin{theorem}[see \cite{ledererRefinedPosterioriError2019}]\label{thm1}
Let $(\boldsymbol{y}, p)$ denote the solutions of equations (\ref{Eq: Stokes weak 1}). Assume that $\boldsymbol{y}_h$ represents the discrete velocity of equations (\ref{Eq: Stokes dis 1}). Then the error estimator $\varepsilon(\boldsymbol{y}_h,\boldsymbol{f})$
satisfies 
\begin{equation}\label{Eq: Stokes reliability}
\|\nabla(\boldsymbol{y}-\boldsymbol{y}_h)\|\leq \varepsilon(\boldsymbol{y}_h,\boldsymbol{f}),
\end{equation}
and
\begin{equation}\label{Eq: Stokes efficiency}
\varepsilon(\boldsymbol{y}_h,\boldsymbol{f})\lesssim \|\nabla(\boldsymbol{y}-\boldsymbol{y}_h)\|+OSC_{2}(\boldsymbol{f},p,\boldsymbol{y}_h),
\end{equation}
where the higher order term is defined as
\begin{align}
\begin{aligned}\label{Def: higher order term}
OSC_{2}(\boldsymbol{f},p,\boldsymbol{y}_h):=&\frac{1}{\nu}(osc(\boldsymbol{f}-\nabla p, \mathcal{T}_h)+osc(\boldsymbol{f}-\nabla p, \mathcal{T}_h,1)\\
& +osc(curl(\boldsymbol{f}+\nu \triangle \boldsymbol{y}_h),\mathcal{T}_h) +osc(\boldsymbol{f}\cdot \tau_E,\mathcal{E}_h)).
\end{aligned}
\end{align}

\end{theorem}

\begin{remark}
The \textit{a posteriori} error estimator $\varepsilon(\boldsymbol{y}_h,\boldsymbol{f})$ is derived from Theorem~5 in  \cite{ledererRefinedPosterioriError2019}, where $``\sigma=\nabla\boldsymbol{u}_h"$ (as indicated in \cite{ledererRefinedPosterioriError2019}). In this context, the term involving the dual norm can be bounded above by the volumetric term (refer to page 16 in \cite{ledererRefinedPosterioriError2019}).
\end{remark}

\subsection{Pressure-robust approximation and improved \textit{a posteriori} error estimation for the optimal control problem}
In this part, we construct \textit{a posteriori} error estimator for the pressure-robust finite element
approximation of the optimal control problem. Compared to Stokes equations, it is not trivial to derive \textit{a posteriori} error bounds for its optimal control formulation, as the latter generates a complex coupled discretized system. Our error estimator consists of three contributions, which are related to the discretization error of the optimality conditions, and the discretization errors of the state and the adjoint equations.

With the velocity reconstruction operator, the pressure-robust finite element approximation for the Stokes optimal control problem (\ref{Eq: cost function weak})-(\ref{Eq: Stokes weak}) is: Find the optimal  $(\boldsymbol{y}_h,p_h,\boldsymbol{u}_h)\in \boldsymbol{V}_h\times Q_h\times \boldsymbol{U}_h$ such that
\begin{align} 
&\min J_h(\boldsymbol{y}_h,\boldsymbol{u}_h):=\frac{1}{2}\|\mathcal{R}_h\boldsymbol{y}_h-\boldsymbol{y_d}\|^2+\frac{\alpha}{2}\|\boldsymbol{u}_h\|^2
 \label{Eq: cost function finite}\\
s.t.&\left\{\begin{aligned}\label{Eq: Stokes finite}
a(\boldsymbol{y}_h,\boldsymbol{v}_h)-b(\boldsymbol{v}_h,p_h) &= (\boldsymbol{u}_h,\mathcal{R}_h\boldsymbol{v}_h), &~\forall \boldsymbol{v}_h\in \boldsymbol{V}_h,\\
b(\boldsymbol{y}_h,\phi_h) &= 0, &~\forall  \phi_h \in Q_h.
\end{aligned}
\right.
\end{align}

\begin{remark}
Compared to the cost functional (\ref{Eq: cost function weak}), we add the reconstruction operator to the discrete velocity $\boldsymbol{y}_h$ in ( \ref{Eq: cost function finite}), following the same procedure as in \cite{merdon2023pressure}. This modification is crucial to ensure that the adjoint equation of the problem is also endowed with a pressure-robust finite element discretization, which will be demonstrated in the following lemma.
\end{remark}

For the discrete  optimal control problem, we can derive the following conclusions regarding the existence and uniqueness of the optimal solution, as well as the first-order optimality conditions.

\begin{lemma}\label{Thm: opt sys dis}
The control problem (\ref{Eq: cost function finite})-(\ref{Eq: Stokes finite}) has a unique optimal solution $(\bar{\boldsymbol{y}}_h, \bar{p}_h,\bar{\boldsymbol{u}}_h) \in \boldsymbol{V}_h\times Q_h\times \boldsymbol{U}_h$, which satisfies the following  optimality condition:
\begin{align}\label{Eq: necessary condition dis}
 (\alpha\bar{\boldsymbol{u}}_h+\mathcal{R}_h\bar{\boldsymbol{z}}_h,\boldsymbol{\mu}_h)=0,\quad \forall \boldsymbol{\mu}_h\in \boldsymbol{U}_h,
\end{align}
and the state equations:
\begin{align}\label{Eq: Stokes finite 1}
\begin{aligned}
a(\bar{\boldsymbol{y}}_h,\boldsymbol{v}_h)-b(\boldsymbol{v}_h,\bar{p}_h) &= (\bar{\boldsymbol{u}}_h,\mathcal{R}_h\boldsymbol{v}_h), &~\forall \boldsymbol{v}_h\in \boldsymbol{V}_h,\\
b(\bar{\boldsymbol{y}}_h,\phi_h) &= 0, &~\forall  \phi_h \in Q_h,
\end{aligned}
\end{align}
and the adjoint equations:
\begin{align}\label{Eq: opt sys dis}
\begin{aligned}
a(\boldsymbol{w}_h,\bar{\boldsymbol{z}}_h)+b(\boldsymbol{w}_h,\bar{r}_h)&= (\mathcal{R}_h\bar{\boldsymbol{y}}_h-\boldsymbol{y_d},\mathcal{R}_h\boldsymbol{w}_h), &~\forall \boldsymbol{w}_h\in \boldsymbol{V}_h,\\
b(\bar{\boldsymbol{z}}_h,\psi_h)&= 0, &~\forall  \psi_h \in Q_h,
\end{aligned}
\end{align}
where $\bar{\boldsymbol{z}}_h\in \boldsymbol{V}_h$ and $\bar{r}_h\in Q_h$ are the optimal adjoint velocity and adjoint pressure respectively.

\end{lemma}
\begin{proof}[Sketch of a proof]
Employing the linear characteristics of the divergence-free reconstruction operator, it is easy to prove that the optimal control problem is continuous, strictly convex, and radially unbounded. By some standard techniques ( see \cite{de2015numerical}, Theorem 3.1 and \cite{hinze2008optimization}, Theorem 1.43), we can establish that problem (\ref{Eq: cost function finite})-(\ref{Eq: Stokes finite}) admits a unique optimal control. And the first-order necessary and sufficient condition can be proved by the sensitivity approach (see \cite{troltzsch2024optimal}).
\end{proof}

Now, we focus on the pressure-robust \textit{a posteriori} error estimation for the discrete optimal control problem.
Confirm that $\bar{\boldsymbol{u}}_h,\bar{\boldsymbol{y}}_h,\bar{\boldsymbol{z}}_h$ represent the optimal control, optimal velocity, and optimal adjoint state variable for the problem (\ref{Eq: cost function finite})-(\ref{Eq: Stokes finite}), respectively.
 Define \textit{a posteriori} error estimator to the optimal control problem (\ref{Eq: cost function weak})-(\ref{Eq: Stokes weak}) and its formulation (\ref{Eq: cost function finite})-(\ref{Eq: Stokes finite}) as
 \begin{align}\label{Def: e gl}
\varepsilon_{pr}:=
\varepsilon_0(\bar{\boldsymbol{z}}_h,\bar{\boldsymbol{u}}_h)+
\varepsilon(\bar{\boldsymbol{y}}_h,\bar{\boldsymbol{u}}_h)+\varepsilon(\bar{\boldsymbol{z}}_h,\mathcal{R}_h\bar{\boldsymbol{y}}_h-\boldsymbol{y_d}),
 \end{align}
where the complete expression of $\varepsilon(\bar{\boldsymbol{y}}_h,\bar{\boldsymbol{u}}_h)$ and $\varepsilon(\bar{\boldsymbol{z}}_h,\mathcal{R}_h\bar{\boldsymbol{y}}_h-\boldsymbol{y_d})$ are presented in equation (\ref{Def: error estimator}). And the definition of $\varepsilon_0(\bar{\boldsymbol{z}}_h,\bar{\boldsymbol{u}}_h)$ is given by
\begin{align}\label{Def: e 0}
\varepsilon_0(\bar{\boldsymbol{z}}_h,\bar{\boldsymbol{u}}_h):=&\|\frac{1}{\alpha}\bar{\boldsymbol{z}}_h+\bar{\boldsymbol{u}}_h\|.
\end{align}

To estimate the \textit{a posteriori} error, we introduce some new variables that will be used . The first one is  the variable $\tilde{\boldsymbol{u}}$ as
\begin{align}\label{Def: u tilde}
\tilde{\boldsymbol{u}}=-\frac{1}{\alpha}\boldsymbol{\bar{z}}_h. 
\end{align}
Then establish $(\tilde{\boldsymbol{y}}, \tilde{p})\in \boldsymbol{V}\times Q$ and $(\tilde{\boldsymbol{z}}, \tilde{r})\in \boldsymbol{V}\times Q$, which solve the following equations,
\begin{align}\label{Eq: opt sys tilde}
\begin{aligned}
a(\tilde{\boldsymbol{y}},\boldsymbol{v})-b(\boldsymbol{v},\tilde{p})&= (\tilde{\boldsymbol{u}},\boldsymbol{v}), &~\forall \boldsymbol{v}\in \boldsymbol{V},\\
b(\tilde{\boldsymbol{y}},\phi)&= 0, &~\forall  \phi \in Q,\\
a(\boldsymbol{w},\tilde{\boldsymbol{z}})+b(\boldsymbol{w},\tilde{r})&= (\tilde{\boldsymbol{y}}-\boldsymbol{y_d},\boldsymbol{w}), &~\forall \boldsymbol{w}\in \boldsymbol{V},\\
b(\tilde{\boldsymbol{z}},\psi)&= 0, &~\forall  \psi \in Q.
\end{aligned}
\end{align}
We also define the solutions $(\hat{\boldsymbol{y}}, \hat{p})\in \boldsymbol{V}\times Q$ and $(\hat{\boldsymbol{z}}, \hat{r})\in \boldsymbol{V}\times Q$ to the system of equations given in
\begin{align}\label{Eq: opt sys hat}
\begin{aligned}
a(\hat{\boldsymbol{y}},\boldsymbol{v})-b(\boldsymbol{v},\hat{p})&= (\bar{\boldsymbol{u}}_h,\boldsymbol{v}), &~\forall \boldsymbol{v}\in \boldsymbol{V},\\
b(\hat{\boldsymbol{y}},\phi)&= 0, &~\forall  \phi \in Q,\\
a(\boldsymbol{w},\hat{\boldsymbol{z}})+b(\boldsymbol{w},\hat{r})&= (\bar{\boldsymbol{y}}_h-\boldsymbol{y_d},\boldsymbol{w}), &~\forall \boldsymbol{w}\in \boldsymbol{V},\\
b(\hat{\boldsymbol{z}},\psi)&= 0, &~\forall  \psi \in Q.
\end{aligned}
\end{align}
Lastly, we introduce $(\check{\boldsymbol{z}}, \check{r})\in \boldsymbol{V}\times Q$, which satisfies the adjoint system outlined in
\begin{align}\label{Eq: adjoint check}
\begin{aligned}
a(\boldsymbol{w},\check{\boldsymbol{z}})+b(\boldsymbol{w},\check{r})&= (\mathcal{R}_h\bar{\boldsymbol{y}}_h-\boldsymbol{y_d},\boldsymbol{w}), &~\forall \boldsymbol{w}\in \boldsymbol{V},\\
b(\check{\boldsymbol{z}},\psi)&= 0, &~\forall  \psi \in Q.
\end{aligned}
\end{align}

The following theorem provides the global upper bounds for control variables, velocities, and the adjoint state. 


\begin{theorem}\label{thm2}
Recall that $(\bar{\boldsymbol{u}},\bar{\boldsymbol{y}},\bar{\boldsymbol{z}})$ is the solution to the problem (\ref{Eq: opt sys}), and $(\bar{\boldsymbol{u}}_h, \bar{\boldsymbol{y}}_h, \bar{\boldsymbol{z}}_h)$ is the solution to (\ref{Eq: necessary condition dis})-(\ref{Eq: opt sys dis}). Then there holds

\begin{align*}
\|\bar{\boldsymbol{u}}-\bar{\boldsymbol{u}}_h\|
+\|\nabla(\bar{\boldsymbol{y}}-\bar{\boldsymbol{y}}_h)\|
+\|\nabla(\bar{\boldsymbol{z}}-\bar{\boldsymbol{z}}_h)\|\lesssim \varepsilon_{pr}+OSC_3,
\end{align*}
where the higher order term is defined as
\begin{align}\label{eqn:OSC_3}
OSC_3=\left(\sum_{T\in\mathcal{T}_h}h_T^6|\bar{\boldsymbol{y}}_h|_{2,T}^2\right)^{1/2}.
\end{align}

\end{theorem}
\begin{proof}

\textit{(i)}~
Employing the auxiliary variable $\tilde{\boldsymbol{u}}$, and using the triangle inequality, we have
\begin{align}\label{Eq: temp 3}
\|\bar{\boldsymbol{u}}-\bar{\boldsymbol{u}}_h\|&\leq \|\bar{\boldsymbol{u}}-\tilde{\boldsymbol{u}}\|+\|\tilde{\boldsymbol{u}}-\bar{\boldsymbol{u}}_h\|\
\end{align}

We then estimate the first term $\|\bar{\boldsymbol{u}}-\tilde{\boldsymbol{u}}\|$ using the optimality condition (\ref{Eq: 1 order necessary optimal condition}) and (\ref{Def: u tilde})
\begin{align}\label{Eq: temp 4}
\begin{aligned}
\alpha\|\bar{\boldsymbol{u}}-\tilde{\boldsymbol{u}}\|^2 &= \alpha(\bar{\boldsymbol{u}}-\tilde{\boldsymbol{u}},\bar{\boldsymbol{u}}-\tilde{\boldsymbol{u}})\\
&= (\tilde{\boldsymbol{u}}-\bar{\boldsymbol{u}},\bar{\boldsymbol{z}}-\bar{\boldsymbol{z}}_h)\\
&=  (\tilde{\boldsymbol{u}}-\bar{\boldsymbol{u}},\bar{\boldsymbol{z}}-\tilde{\boldsymbol{z}})+ (\tilde{\boldsymbol{u}}-\bar{\boldsymbol{u}},\tilde{\boldsymbol{z}}-\bar{\boldsymbol{z}}_h)
\end{aligned}
\end{align}

For the term  $(\tilde{\boldsymbol{u}}-\bar{\boldsymbol{u}},\bar{\boldsymbol{z}}-\tilde{\boldsymbol{z}})$, we refer to the technique in \cite{allendes2019adaptive} for estimation, outlined here only for the sake of completeness.
Subtracting (\ref{Eq: opt sys tilde}) from (\ref{Eq: opt sys}), and selecting the test functions $\boldsymbol{v}=\bar{\boldsymbol{z}}-\tilde{\boldsymbol{z}}$, $\phi=\bar{r}-\tilde{r}$, $\boldsymbol{w}=\bar{\boldsymbol{y}}-\tilde{\boldsymbol{y}}$ and $\psi=\bar{p}-\tilde{p}$, it derives the system
\begin{align*}
\begin{aligned}
a(\bar{\boldsymbol{y}}-\tilde{\boldsymbol{y}},\bar{\boldsymbol{z}}-\tilde{\boldsymbol{z}})-b(\bar{\boldsymbol{z}}-\tilde{\boldsymbol{z}},\bar{p}-\tilde{p})&= (\bar{\boldsymbol{u}}-\tilde{\boldsymbol{u}},\bar{\boldsymbol{z}}-\tilde{\boldsymbol{z}}),\\
b(\bar{\boldsymbol{y}}-\tilde{\boldsymbol{y}},\bar{r}-\tilde{r})&= 0,\\
a(\bar{\boldsymbol{y}}-\tilde{\boldsymbol{y}},\bar{\boldsymbol{z}}-\tilde{\boldsymbol{z}})+b(\bar{\boldsymbol{y}}-\tilde{\boldsymbol{y}},\bar{r}-\tilde{r})&= (\bar{\boldsymbol{y}}-\tilde{\boldsymbol{y}},\bar{\boldsymbol{y}}-\tilde{\boldsymbol{y}}),\\
b(\bar{\boldsymbol{z}}-\tilde{\boldsymbol{z}},\bar{p}-\tilde{p})&= 0.
\end{aligned}
\end{align*}
This implies that
\begin{align}\label{Eq: temp 5}
(\bar{\boldsymbol{u}}-\tilde{\boldsymbol{u}},\bar{\boldsymbol{z}}-\tilde{\boldsymbol{z}})
=(\bar{\boldsymbol{y}}-\tilde{\boldsymbol{y}},\bar{\boldsymbol{y}}-\tilde{\boldsymbol{y}})=\|\bar{\boldsymbol{y}}-\tilde{\boldsymbol{y}}\|^2.
\end{align}
By substituting (\ref{Eq: temp 5}) into (\ref{Eq: temp 4}), we get
\begin{align*}
\alpha\|\bar{\boldsymbol{u}}-\tilde{\boldsymbol{u}}\|^2 
=&  (\tilde{\boldsymbol{u}}-\bar{\boldsymbol{u}},\bar{\boldsymbol{z}}-\tilde{\boldsymbol{z}})+ (\tilde{\boldsymbol{u}}-\bar{\boldsymbol{u}},\tilde{\boldsymbol{z}}-\bar{\boldsymbol{z}}_h)\\
=& -\|\bar{\boldsymbol{y}}-\tilde{\boldsymbol{y}}\|^2+ (\tilde{\boldsymbol{u}}-\bar{\boldsymbol{u}},\tilde{\boldsymbol{z}}-\bar{\boldsymbol{z}}_h)\\
\leq& (\tilde{\boldsymbol{u}}-\bar{\boldsymbol{u}},\tilde{\boldsymbol{z}}-\bar{\boldsymbol{z}}_h).
\end{align*}
Then, using H$\ddot{o}$lder inequality it has 
\begin{align}\label{Eq: temp 2}
\begin{aligned}
\alpha\|\bar{\boldsymbol{u}}-\tilde{\boldsymbol{u}}\|&\lesssim \|\tilde{\boldsymbol{z}}-\bar{\boldsymbol{z}}_h\|\\
 &\lesssim \|\tilde{\boldsymbol{z}}-\hat{\boldsymbol{z}}\|+\|\hat{\boldsymbol{z}}-\check{\boldsymbol{z}}\|
 +\|\check{\boldsymbol{z}}-\bar{\boldsymbol{z}}_h\|.
 \end{aligned}
\end{align}

Next, we first estimate the second term on the right-hand side of (\ref{Eq: temp 2}).
Subtracting equation (\ref{Eq: adjoint check}) from the adjoint equations of (\ref{Eq: opt sys hat}) yields
\begin{align*}
\begin{aligned}
a(\boldsymbol{w},\hat{\boldsymbol{z}}-\check{\boldsymbol{z}})+b(\boldsymbol{w},\hat{r}-\check{r})&= (\bar{\boldsymbol{y}}_h-\mathcal{R}_h\bar{\boldsymbol{y}}_h,\boldsymbol{w}), &~\forall \boldsymbol{w}\in \boldsymbol{V},\\
b(\hat{\boldsymbol{z}}-\check{\boldsymbol{z}},\psi)&= 0, &~\forall  \psi \in Q.
\end{aligned}
\end{align*} 
Setting $\boldsymbol{w}=\hat{\boldsymbol{z}}-\check{\boldsymbol{z}}$ and $\psi=\hat{r}-\check{r}$, we find
\begin{align}\label{Eq: temp}
a(\hat{\boldsymbol{z}}-\check{\boldsymbol{z}},\hat{\boldsymbol{z}}-\check{\boldsymbol{z}})= (\bar{\boldsymbol{y}}_h-\mathcal{R}_h\bar{\boldsymbol{y}}_h,\hat{\boldsymbol{z}}-\check{\boldsymbol{z}}). 
\end{align} 
Utilizing (\ref{Eq: reconstruction operator def2}), we can transform equation (\ref{Eq: temp}) into
\begin{align*}
a(\hat{\boldsymbol{z}}-\check{\boldsymbol{z}},\hat{\boldsymbol{z}}-\check{\boldsymbol{z}})= (\bar{\boldsymbol{y}}_h-\mathcal{R}_h\bar{\boldsymbol{y}}_h,(\hat{\boldsymbol{z}}-\check{\boldsymbol{z}})-\emph{I}_{0}(\hat{\boldsymbol{z}}-\check{\boldsymbol{z}})),
\end{align*}
where the definition of  $\emph{ I}_{0}$  is given by $\emph{I}_{k}$ in section 2.2 for $k=0$.

From estimates (\ref{Eq: interpolation estimate}) and (\ref{Eq: 1-R estimate}), the above equation can be deduced that
\begin{align*}
|\hat{\boldsymbol{z}}-\check{\boldsymbol{z}}|_1^2 &\lesssim \sum_{T\in\mathcal{T}_h}\|\bar{\boldsymbol{y}}_h-\mathcal{R}_h\bar{\boldsymbol{y}}_h\|_{0,T}
\|(\hat{\boldsymbol{z}}-\check{\boldsymbol{z}})-\emph{I}_{0}(\hat{\boldsymbol{z}}-\check{\boldsymbol{z}})\|_{0,T}\\
&\lesssim OSC_3\|\hat{\boldsymbol{z}}-\check{\boldsymbol{z}}\|_{1}. 
\end{align*}
Incorporating the Poincar\'{e} inequality, it follows that
\begin{align*}
\|\hat{\boldsymbol{z}}-\check{\boldsymbol{z}}\|_{1}^2
&\lesssim OSC_3\|\hat{\boldsymbol{z}}-\check{\boldsymbol{z}}\|_{1}. 
\end{align*}  
Thereby, we have
\begin{align}\label{Eq: temp 1}
\|\hat{\boldsymbol{z}}-\check{\boldsymbol{z}}\|\leq\|\hat{\boldsymbol{z}}-\check{\boldsymbol{z}}\|_{1}
\lesssim OSC_3. 
\end{align}   
 
Referring to (\ref{Eq: temp 2}), we now present a complete estimate, utilizing
Lemma~\ref{Lem: stability estimate}, equation (\ref{Eq: temp 1}), the Poincar\'{e} inequality and Theorem~\ref{thm1}.
\begin{align*}
\alpha\|\bar{\boldsymbol{u}}-\tilde{\boldsymbol{u}}\|\lesssim& \|\tilde{\boldsymbol{z}}-\bar{\boldsymbol{z}}_h\|\\
 \leq&\|\tilde{\boldsymbol{z}}-\hat{\boldsymbol{z}}\|+\|\hat{\boldsymbol{z}}-\check{\boldsymbol{z}}\|
 +\|\check{\boldsymbol{z}}-\bar{\boldsymbol{z}}_h\|\\
\lesssim& \|\tilde{\boldsymbol{z}}-\hat{\boldsymbol{z}}\|+\|\hat{\boldsymbol{z}}-\check{\boldsymbol{z}}\|+\|\nabla(\check{\boldsymbol{z}}-\bar{\boldsymbol{z}} _h)\|\\
\lesssim& \|\tilde{\boldsymbol{y}}-\bar{\boldsymbol{y}}_h\| + OSC_3+\varepsilon(\bar{\boldsymbol{z}}_h,\mathcal{R}_h\bar{\boldsymbol{y}}_h-\boldsymbol{y_d})\\
\leq& \|\tilde{\boldsymbol{y}}-\hat{\boldsymbol{y}}\|+\|\hat{\boldsymbol{y}}-\bar{\boldsymbol{y}}_h\| + OSC_3+ \varepsilon(\bar{\boldsymbol{z}}_h,\mathcal{R}_h\bar{\boldsymbol{y}}_h-\boldsymbol{y_d})\\
\leq& \|\tilde{\boldsymbol{y}}-\hat{\boldsymbol{y}}\|+\|\nabla(\hat{\boldsymbol{y}}-\bar{\boldsymbol{y}}_h)\| + OSC_3+ \varepsilon(\bar{\boldsymbol{z}}_h,\mathcal{R}_h\bar{\boldsymbol{y}}_h-\boldsymbol{y_d})\\
\lesssim& \|\tilde{\boldsymbol{u}}-\bar{\boldsymbol{u}}_h\|+\varepsilon(\bar{\boldsymbol{y}}_h,\bar{\boldsymbol{u}}_h)+
OSC_3+\varepsilon(\bar{\boldsymbol{z}}_h,\mathcal{R}_h\bar{\boldsymbol{y}}_h-\boldsymbol{y_d}).
\end{align*}

Notice that the second term on the right-hand side of (\ref{Eq: temp 3}) is the error estimator 
defined in (\ref{Def: e 0}), based on the definition of $\tilde{\boldsymbol{u}}$ provided in (\ref{Def: u tilde}).
Thus, we can finally arrive at
\begin{align*}
\|\bar{\boldsymbol{u}}-\bar{\boldsymbol{u}}_h\|\lesssim   (1+\alpha^{-1})\varepsilon_0(\bar{\boldsymbol{z}}_h,\bar{\boldsymbol{u}}_h)
+\alpha^{-1}\left(\varepsilon(\bar{\boldsymbol{y}}_h,\bar{\boldsymbol{u}}_h)
+OSC_3
+\varepsilon(\bar{\boldsymbol{z}}_h,\mathcal{R}_h\bar{\boldsymbol{y}}_h-\boldsymbol{y_d})\right).
\end{align*}

\textit{(ii)}~Using the conclusion from \textit{(i)}, and combining it with Lemma~\ref{Lem: stability estimate} and Theorem~\ref{thm1}, the proof can be easily accomplished
\begin{align*} 
\|\nabla(\bar{\boldsymbol{y}}-\bar{\boldsymbol{y}}_h)\|\leq&
\|\nabla(\bar{\boldsymbol{y}}-\hat{\boldsymbol{y}})\|+\|\nabla(\hat{\boldsymbol{y}}-\bar{\boldsymbol{y}}_h)\|\\
\lesssim&  \|\bar{\boldsymbol{u}}-\bar{\boldsymbol{u}}_h\|+ \varepsilon(\bar{\boldsymbol{y}}_h,\bar{\boldsymbol{u}}_h)\\
\leq&  (1+\alpha^{-1})\left[\varepsilon_0(\bar{\boldsymbol{z}}_h,\bar{\boldsymbol{u}}_h)
+\varepsilon(\bar{\boldsymbol{y}}_h,\bar{\boldsymbol{u}}_h)\right]\\
&+\alpha^{-1}\left(\varepsilon(\bar{\boldsymbol{z}}_h,\mathcal{R}_h\bar{\boldsymbol{y}}_h-\boldsymbol{y_d})
+OSC_3\right).
\end{align*}

\textit{(iii)}~ 
Similar to the proof of \textit{(ii)}, additionally employing (\ref{Eq: temp 1}), we can deduce that
\begin{align*}
\|\nabla(\bar{\boldsymbol{z}}-\bar{\boldsymbol{z}}_h)\|\leq&
\|\nabla(\bar{\boldsymbol{z}}-\hat{\boldsymbol{z}})\|+
\|\nabla(\hat{\boldsymbol{z}}-\check{\boldsymbol{z}})\|+\|\nabla(\check{\boldsymbol{z}}-\bar{\boldsymbol{z}}_h)\|\\
\lesssim&  \|\bar{\boldsymbol{y}}-\bar{\boldsymbol{y}}_h\|+OSC_3+ \varepsilon(\bar{\boldsymbol{z}}_h,\mathcal{R}_h\bar{\boldsymbol{y}}_h-\boldsymbol{y_d})\\
\leq&  (1+\alpha^{-1})\left(\varepsilon_0(\bar{\boldsymbol{z}}_h,\bar{\boldsymbol{u}}_h)+
\varepsilon(\bar{\boldsymbol{y}}_h,\bar{\boldsymbol{u}}_h)
+\varepsilon(\bar{\boldsymbol{z}}_h,\mathcal{R}_h
\bar{\boldsymbol{y}}_h-\boldsymbol{y_d})\right)
+\alpha^{-1}OSC_3.
\end{align*}
\end{proof}

The proof of the global lower bound for the \textit{a posteriori} error estimate is straightforward, as outlined below.
\begin{theorem}\label{thm4}
Let $\bar{\boldsymbol{u}},\bar{\boldsymbol{y}},\bar{\boldsymbol{z}}$ be the solutions of problem (\ref{Eq: opt sys}), and $\bar{\boldsymbol{u}}_h, \bar{\boldsymbol{y}}_h, \bar{\boldsymbol{z}}_h$ be the solutions of (\ref{Eq: necessary condition dis})-(\ref{Eq: opt sys dis}), then the error indicator $\varepsilon_{pr}$ defined in (\ref{Def: e gl}) satisfies the following estimate
\begin{align*}
\varepsilon_{pr}\lesssim &\|\bar{\boldsymbol{u}}-\bar{\boldsymbol{u}}_h\|+\|\nabla(\bar{\boldsymbol{y}}-\bar{\boldsymbol{y}}_h)\|+\|\nabla(\bar{\boldsymbol{z}}-
\bar{\boldsymbol{z}}_h)\|\\
&+OSC_2(\bar{\boldsymbol{u}}_h,\hat{p},\bar{\boldsymbol{y}}_h)+OSC_2(\bar{\boldsymbol{y}}_h-\boldsymbol{y_d},\check{r},\bar{\boldsymbol{z}}_h)+OSC_3,
\end{align*}
where the last three terms on the right-hand side of the inequality are of higher order, while $OSC_2$ and $OSC_3$ are defined in~\eqref{Def: higher order term} and~\eqref{eqn:OSC_3}, respectively.
The pressure $\hat{p}$ and $\check{r}$ are the solutions of (\ref{Eq: opt sys hat}) and (\ref{Eq: adjoint check}), respectively.
\end{theorem}
\begin{proof}

Recall that $\varepsilon_{pr}$ consists of three terms: $\varepsilon_0(\bar{\boldsymbol{z}}_h,\bar{\boldsymbol{u}}_h)$, 
$\varepsilon(\bar{\boldsymbol{y}}_h, \bar{\boldsymbol{u}}_h)$, and $\varepsilon(\bar{\boldsymbol{z}}_h,\bar{\boldsymbol{y}}_h)$. Next, we prove the efficiency of each term separately.

For $\varepsilon_0(\bar{\boldsymbol{z}}_h,\bar{\boldsymbol{u}}_h)$, we have
\begin{align*}
\varepsilon_0(\bar{\boldsymbol{z}}_h,\bar{\boldsymbol{u}}_h)=&\|\frac{1}{\alpha}\bar{\boldsymbol{z}}_h+\bar{\boldsymbol{u}}_h\|\\
\leq&  \|\frac{1}{\alpha}(\bar{\boldsymbol{z}}_h-\bar{\boldsymbol{z}})\|+\|\frac{1}{\alpha}\bar{\boldsymbol{z}}+\bar{\boldsymbol{u}}_h\|\\
=&  \|\frac{1}{\alpha}(\bar{\boldsymbol{z}}_h-\bar{\boldsymbol{z}})\|+\|\bar{\boldsymbol{u}}-\bar{\boldsymbol{u}}_h\|\\ 
\lesssim& \alpha^{-1}\|\nabla(\bar{\boldsymbol{z}}_h-\bar{\boldsymbol{z}})\|+\|\bar{\boldsymbol{u}}-\bar{\boldsymbol{u}}_h\|.
\end{align*}

The proofs for the remaining two are mainly based on Lemma~\ref{Lem: stability estimate}.
From (\ref{Eq: Stokes efficiency}) and (\ref{Eq: temp 1}), it can be obtained 
\begin{align*}
\varepsilon(\bar{\boldsymbol{y}}_h,\bar{\boldsymbol{u}}_h)\lesssim&  \|\nabla(\hat{\boldsymbol{y}}-\bar{\boldsymbol{y}}_h)\|+OSC_2(\bar{\boldsymbol{u}}_h,\hat{p},\bar{\boldsymbol{y}}_h)\\
\leq&  \| \nabla(\hat{\boldsymbol{y}}-\bar{\boldsymbol{y}})\|+\|\nabla(\bar{\boldsymbol{y}}-\bar{\boldsymbol{y}}_h)\|+OSC_2(\bar{\boldsymbol{u}}_h,\hat{p},\bar{\boldsymbol{y}}_h)\\
\lesssim& \|\bar{\boldsymbol{u}}_h-\bar{\boldsymbol{u}}\|+\|\nabla(\bar{\boldsymbol{y}}-\bar{\boldsymbol{y}}_h)\|+OSC_2(\bar{\boldsymbol{u}}_h,\hat{p},\bar{\boldsymbol{y}}_h).
\end{align*}
and
\begin{align*}
\varepsilon(\bar{\boldsymbol{z}}_h,\bar{\boldsymbol{y}}_h)  \lesssim & \| \nabla(\check{\boldsymbol{z}}-\bar{\boldsymbol{z}}_h)\|+OSC_2(\bar{\boldsymbol{y}}_h-\boldsymbol{y_d},\check{r},\bar{\boldsymbol{z}}_h)\\
\leq&  \| \nabla(\check{\boldsymbol{z}}-\hat{\boldsymbol{z}})\|+\| \nabla(\hat{\boldsymbol{z}}-\bar{\boldsymbol{z}})\|+\| \nabla(\bar{\boldsymbol{z}}-\bar{\boldsymbol{z}}_h)\|+OSC_2(\bar{\boldsymbol{y}}_h-\boldsymbol{y_d},\check{r},\bar{\boldsymbol{z}}_h)\\
\lesssim& OSC_3+\|\nabla(\bar{\boldsymbol{y}}_h-\bar{\boldsymbol{y}})\|+\|\nabla(\bar{\boldsymbol{z}}-\bar{\boldsymbol{z}}_h)\|+OSC_2(\bar{\boldsymbol{y}}_h-\boldsymbol{y_d},\check{r},\bar{\boldsymbol{z}}_h).
\end{align*}

\end{proof}

\begin{remark}
The proofs of Theorems \ref{thm2} and \ref{thm4} indicate that both the reliability and efficiency constants depend on $\alpha$.
\end{remark}

\begin{remark}[Higher-order oscillation terms]
The \textit{a posteriori error} estimates in Theorems \ref{thm2} and \ref{thm4} 
involve several oscillation
terms, denoted by $OSC_1$, $OSC_2$, and $OSC_3$.
These terms arise from data approximation, reconstruction, and interpolation procedures,
and are standard in residual-based \textit{a posteriori} error analysis.
They are of higher order with respect to the dominant
residual estimator and vanish under mesh refinement. In particular, for sufficiently smooth
solutions, they do not affect the asymptotic convergence rate of the \textit{a posteriori} error estimates.
\end{remark}

\section{Numerical example}
In this section, we provide three numerical examples in two dimensions to validate the properties of the estimator discussed in Theorems~\ref{thm3},\ref{thm2}, and \ref{thm4}, using uniform and adaptive refinement meshes. In order to show the effectiveness of the global error estimator defined in (\ref{Def: e gl}), we set $E_h=\sqrt{\|\bar{\boldsymbol{u}}-\bar{\boldsymbol{u}}_h\|^2+\|\nabla(\bar{\boldsymbol{y}}-\bar{\boldsymbol{y}}_h)\|^2+\|\nabla(\bar{\boldsymbol{z}}-\bar{\boldsymbol{z}}_h)\|^2}$ and introduce the effectivity indices, defined as
$\kappa_{pr}=\varepsilon_{pr}/E_h$ and $\kappa_c=\hat{\varepsilon}/E_h$, which is the ratio between the global error estimator and the approximation error. Here we substitute the solutions of (\ref{Eq: necessary condition dis})-(\ref{Eq: opt sys dis}) into (\ref{Def: e gl}) and (\ref{Def: error estimator classical}), respectively, to obtain $\varepsilon_{pr}$ and $\hat{\varepsilon}$.
According to Theorems~\ref{thm2} and \ref{thm4}, the effectivity index $\kappa_{pr}$ is bounded from above and below and independent of $h$ and pressure. However, from Theorem~\ref{thm3}, The efficiency analysis of the error estimator $\hat{\varepsilon}$ involves the pressure error (and thus the effectivity index 
$\kappa_c$ depends on the pressure); therefore, it is not suitable as an estimator for the velocity error.

All these numerical examples have been implemented on a 2.4~GHz dual-core processor with 8~G RAM by MATLAB. The matrices were assembled with exact precision. The computation of the right-hand side vectors and the approximation errors has been performed using a quadrature formula that is exact for polynomials up to the eighth degree. For the resolution of the ensuing global linear systems, distinct methodologies are employed for Examples 1, 2, and 3. In Example 1, homogeneous boundary conditions are imposed. When $\alpha>0$, the derivation of an exact solution becomes intractable. To overcome this challenge, a numerical solution obtained on a highly refined grid, characterized by a substantial number of unknowns, serves as a surrogate for the exact solution. Consequently, for Example 1, the backslash operator of MATLAB is utilized for systems with a limited number of unknowns, specifically when the count is below the threshold of 5000. In contrast, for systems with a more substantial number of unknowns, exceeding the limit of 5000, the Uzawa algorithm, with a step size of $0.1$, is used to obtain the solution. In Example 2 and 3, non-homogeneous boundary conditions are applied. The exact solution can be articulated by incorporating an additional source term on the right-hand side of the equation. Given that the number of unknowns in Example 2 and 3 is not excessively large, the backslash operator is applied directly.

\noindent{\bf Example 1.} The first example, taken from \cite{merdon2023pressure}, aims to investigate the convergence behavior of the discrete problem and the new \textit{a posteriori} error estimator on a uniform refinement mesh. This example studies the prescribed polynomial velocity solution 
\begin{align*}
\bar{\boldsymbol{y}}(x,y)=curl\big(x^4(x-1)^4y^4(y-1)^4\big)
\end{align*}
and null pressure solution of the Stokes problem on the unit square, where $curl(v)=(-\partial v/\partial y,\partial v/\partial x)$. Then, the state and control functions $(\bar{\boldsymbol{y}},\bar{\boldsymbol{u}})$, with $\bar{\boldsymbol{u}}=\nu\Delta\bar{\boldsymbol{y}}$, minimize the optimal control problem (\ref{Eq: cost function}) and (\ref{Eq: Stokes}) for $\alpha=0$ and $\boldsymbol{y}_d=\bar{\boldsymbol{y}}$. For $\alpha>0$, the exact minimizer is unknown, and a reference solution is computed on a very fine mesh shown in Fig.~\ref{fig:Reference solutions for ex1}.

\begin{figure}[htbp]
  \centering
  \includegraphics[width=0.325\textwidth]{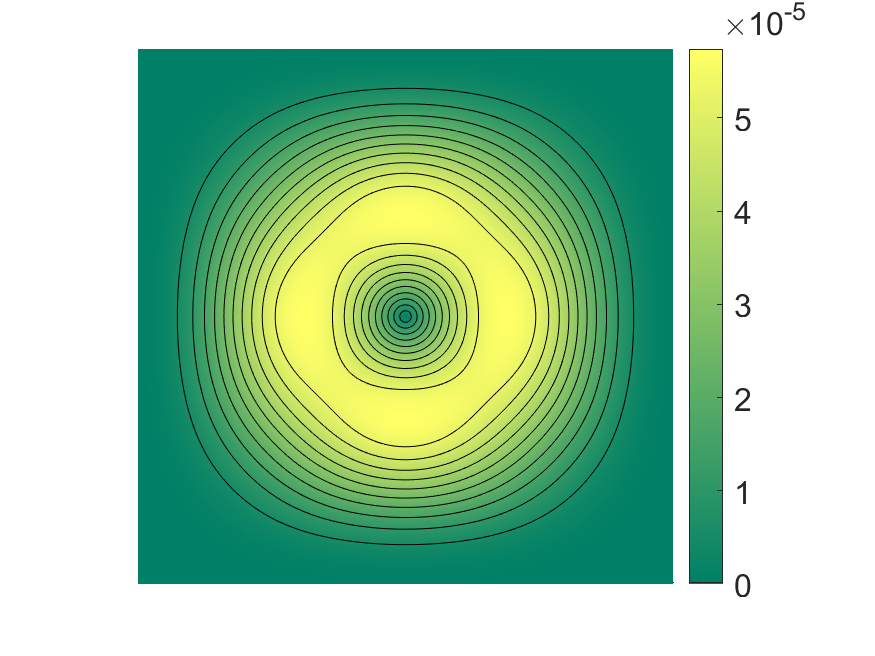}
  \includegraphics[width=0.325\textwidth]{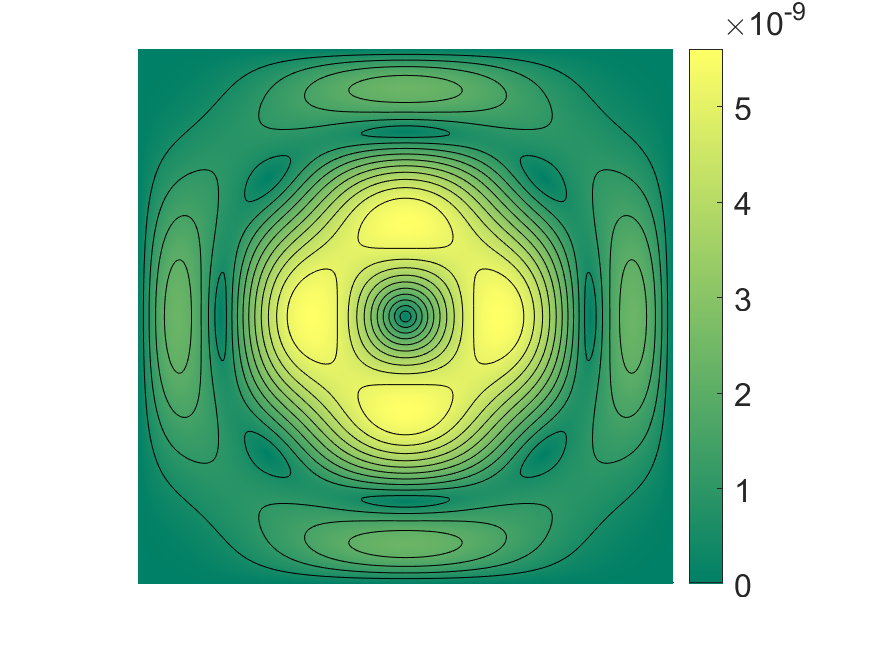}
  \includegraphics[width=0.325\textwidth]{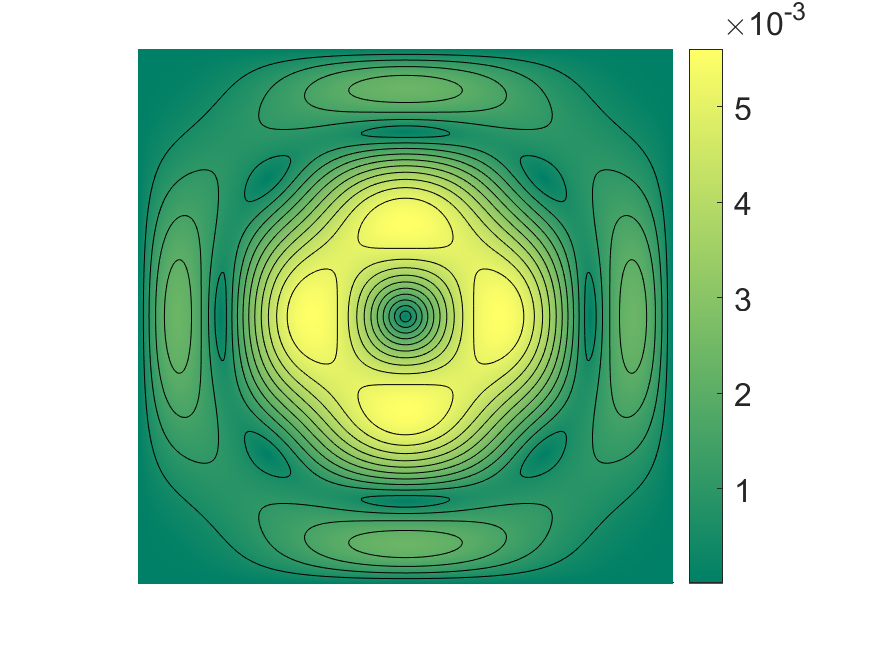}\\
  \includegraphics[width=0.325\textwidth]{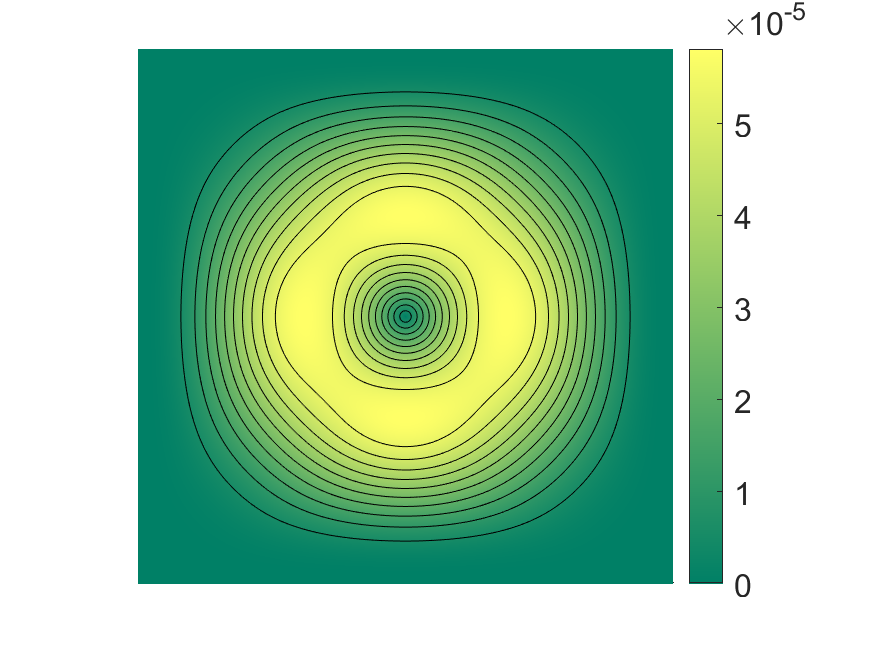}
  \includegraphics[width=0.325\textwidth]{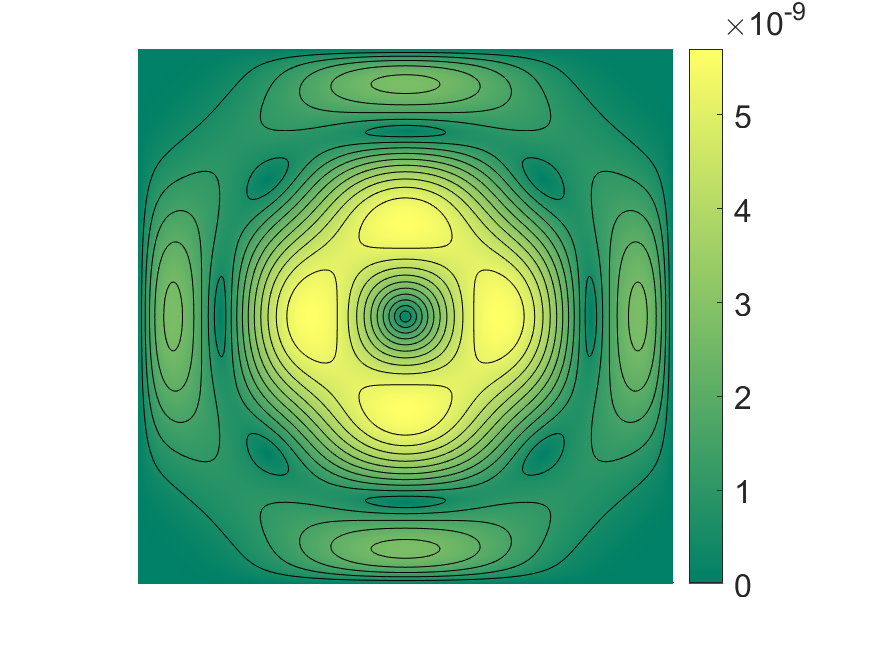}
  \includegraphics[width=0.325\textwidth]{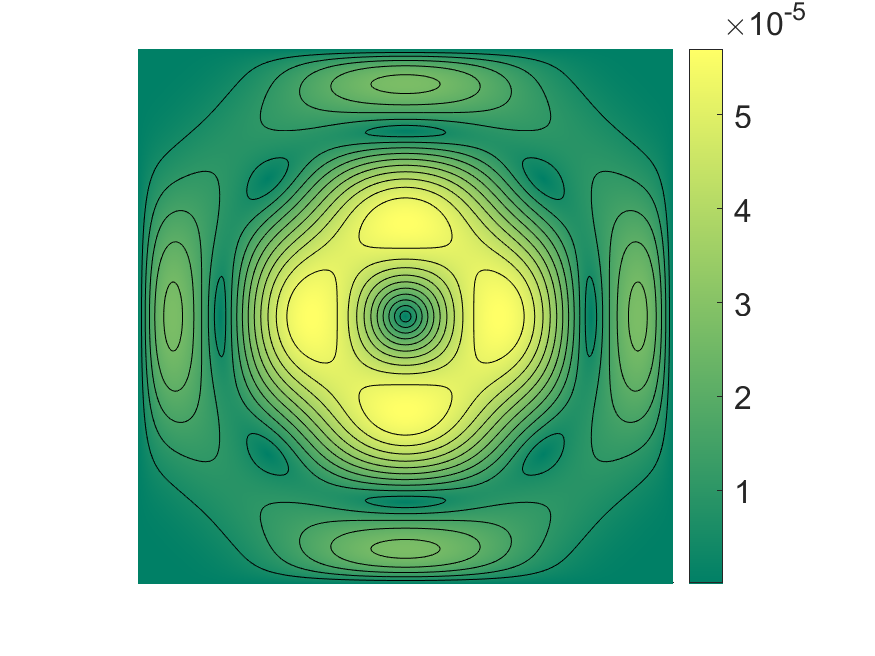}
\caption{Reference solutions for $\boldsymbol{y}$ (left), $\boldsymbol{z}$ (middle), and $\boldsymbol{u}$ (right) of (\ref{Eq: opt sys}) with $d.o.f=2363396$ for $(\nu,\alpha)=(1,10^{-6})$ in top row and $(\nu,\alpha)=(10^{-2},10^{-4})$ in bottom row in Example 1.}
\label{fig:Reference solutions for ex1}
\end{figure}

Tab.~\ref{tab:error ex1} presents a comparison between the approximation error $E_h$ and the \textit{a posteriori} error estimator for the parameter pairs $(\nu,\alpha)=(1,10^{-6})$ and $(10^{-2},10^{-4})$. The findings reveal that the approximation error and
the error estimator maintain the optimal convergence rate for $\nu = 1$ and displays slightly suboptimal behavior pre-asymptotically for $\nu=10^{-2}$ , relative to the degree of freedom ($d.o.f$), as indicated by $N$ in Fig.~\ref{fig:convergence for ex1} and \ref{fig:convergence order ex2}. The effectivity index approaches a constant value, which is independent of the mesh size $h$. These results align with Theorems~\ref{thm2} and \ref{thm4}. For a more intuitive representation of the convergence behavior across varying values of $\nu$ and $\alpha$, see Fig.~\ref{fig:convergence for ex1}. In this context, the various components of the \textit{a posteriori} error estimator share the same convergence rate, with the exception that the term $\varepsilon_0(\bar{\boldsymbol{z}}_h,\bar{\boldsymbol{u}}_h)$ occasionally demonstrates superconvergence behavior. However, this sporadic superconvergence does not influence the overall rate of convergence.

\begin{table}[htbp]
\caption{The error of the discrete problem 
and the \textit{a posteriori} error estimator in Example 1.} \label{tab:error ex1}
\begin{tabular*}{\hsize}{@{}@{\extracolsep{\fill}}cc|cc|cc|c@{}}
    \hline
   $(\nu,\alpha)$&$d.o.f$          &  $E_h$ & order    & $\varepsilon_{pr}$ & order
   &      $\kappa_{pr}$ \\
   \hline
    \hline
   $(1,10^{-6})$&$180$   &  $1.025$E-3  &  ---     &   $1.414$E-2  &---      &  $13.79$\\

                &$644$   &  $7.726$E-4  & $0.221$  &   $6.6216$E-3  & $0.595$ &  $8.563$\\

               &$2436$  &  $2.187$E-4  & $0.948$  &   $1.889$E-3  & $0.941$ &  $8.638$\\

               &$9476$  &  $5.664$E-5  & $0.994$  &   $4.742$E-4  & $1.017$ &  $8.372$\\

               &$37380$ &  $1.354$E-5  & $1.042$  &   $1.168$E-4  & $1.020$ &  $8.623$\\
     
               &$148484$ &  $3.529$E-6  & $0.975$  &   $2.968$E-5  & $0.993$ &  $8.409$\\
    \hline
    \hline
$(10^{-2},10^{-4})$&$180$   &  $1.368$E-4  &  ---     &   $2.644$E-2  &---      &  $193.2$\\

                   &$644$   &  $1.890$E-4  & $-0.253$  &   $3.236$E-2  & $-0.158$ &  $171.2$\\

                   &$2436$  &  $5.346$E-5  & $0.949$  &   $5.073$E-3  & $1.392$ &  $94.89$\\

                   &$9476$  &  $8.137$E-6  & $1.385$  &   $6.655$E-4  & $1.495$ &  $81.78$\\

                   &$37380$ &  $1.785$E-6 & $1.105$  &   $1.636$E-4 & $1.022$ &  $91.69$\\
     
                   &$148484$ &  $4.683$E-7 & $0.970$  &   $4.452$E-5 & $0.943$ &  $95.06$\\
    \hline
\end{tabular*}
\end{table}

\begin{figure}[htbp]
  \centering
  \includegraphics[width=0.4\textwidth]{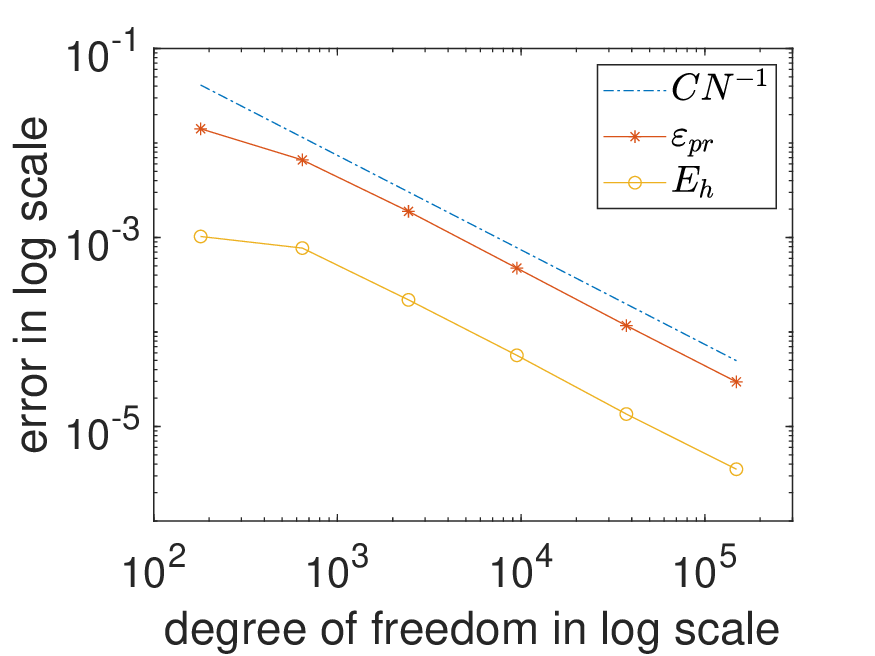}
  \includegraphics[width=0.4\textwidth]{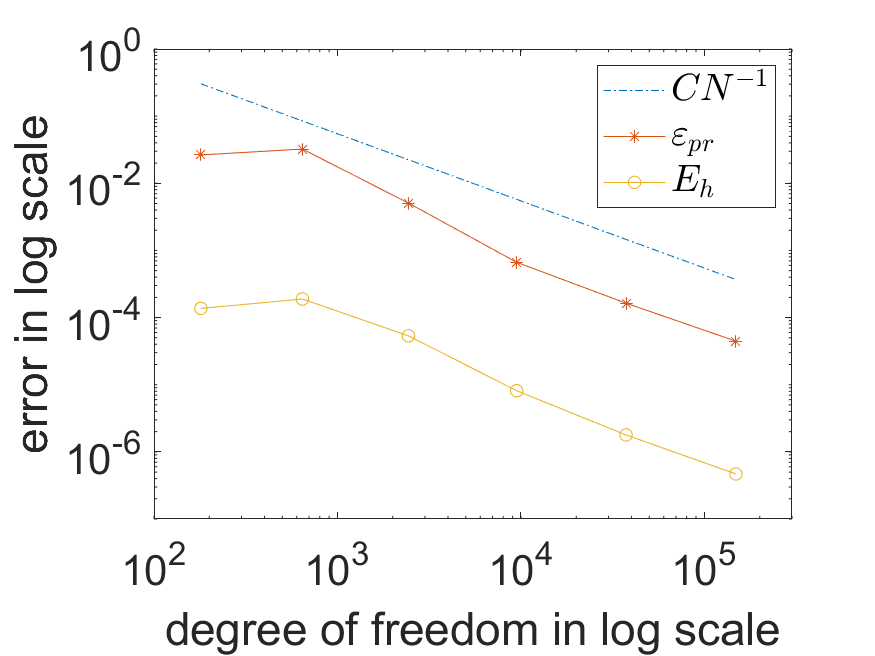}\\
  \includegraphics[width=0.4\textwidth]{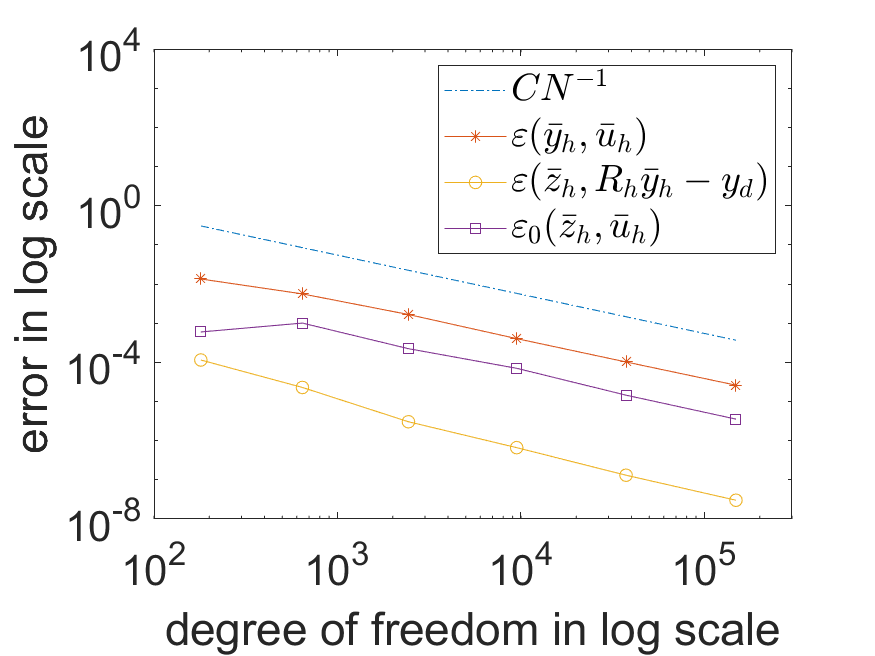}
  \includegraphics[width=0.4\textwidth]{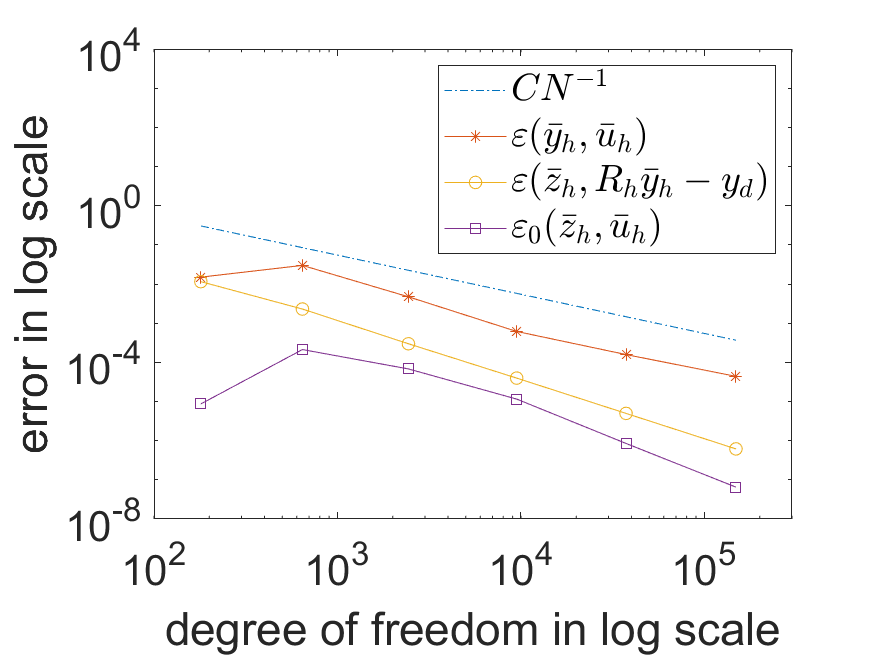}
\caption{Convergence rates (top row) and its individual contributions (bottom row) with $(\nu,\alpha)=(1,10^{-6})$ on the left and $(\nu,\alpha)=(10^{-2},10^{-4})$ on the right in Example 1.}
\label{fig:convergence for ex1}
\end{figure}

\noindent{\bf Example 2.} This example, taken from
\cite{Verfurth1996}, is intended to show the convergence properties of the 
discrete problem (\ref{Eq: cost function finite})-(\ref{Eq: Stokes finite}) and the performance of the \textit{a posteriori} error estimator on an adaptively refined mesh.
Let $\Omega$ be the L-shape domain $(-1,1)^2 \backslash[0,1)\times
(-1,0]$ and $\nu=1$. Then, use $(\rho,
\varphi)$ to denote the polar coordinates. The state functions $\bar{\boldsymbol{y}}=(\bar{y}_1,\bar{y}_2)$ and $\bar{p}$ with singularity at the origin are as follows
\begin{align*}
\bar{y}_1(\rho,\varphi)&=\rho^\lambda((1+\lambda)sin(\varphi)\Psi(\varphi)
+cos(\varphi)\Psi^{\prime}(\varphi)),\\
\bar{y}_2(\rho,\varphi)&=\rho^\lambda(sin(\varphi)\Psi^{\prime}(\varphi)
-(1+\lambda)cos(\varphi)\Psi(\varphi)),\\
\bar{p}(\rho,\varphi)&=-\rho^{\lambda-1}[(1+\lambda)^2\Psi^{\prime}(\varphi)
+\Psi^{\prime\prime\prime}]/(1-\lambda),
\end{align*}
where
\begin{align*}
\Psi(\varphi)&= \sin ((1+\lambda) \varphi) \cos (\lambda \omega)
/(1+\lambda)-\cos ((1+\lambda) \varphi) \\
&\quad-\sin ((1-\lambda) \varphi)
\cos (\lambda \omega) /(1-\lambda)+\cos ((1-\lambda) \varphi),\\
\omega&= \frac{3 \pi}{2}.
\end{align*}
The exponent $\lambda$ is the smallest positive solution of
\begin{align*}
\sin (\lambda \omega)+\lambda \sin (\omega)=0,
\end{align*}
thereby, $\lambda\approx 0.54448373678246$. The adjoint state functions are chosen to be $\bar{\boldsymbol{z}}=0.01\bar{\boldsymbol{y}}$ and $\bar{r}=-0.01\bar{p}$. The control function is given by $\bar{\boldsymbol{u}}=-\bar{\boldsymbol{z}}/\alpha$ with $\alpha=0.01$. The state functions hold a non-homogeneous boundary condition. An extra known source term, denoted by $\boldsymbol{u}_0$ and determined from the exact solutions, will be added to the state equations. In detail, the first state equations in (\ref{Eq: Stokes finite 1}) can be rewriteen as
\begin{align*}
a(\bar{\boldsymbol{y}}_h,\boldsymbol{v}_h)-b(\boldsymbol{v}_h,\bar{p}_h) &= (\bar{\boldsymbol{u}}_h-\boldsymbol{u}_0,\mathcal{R}_h\boldsymbol{v}_h).
\end{align*}
Despite this example featuring nonhomogeneous boundary conditions and additional source terms, which fall outside the scope of the theoretical analysis presented in this paper, the performance of the \textit{a posteriori} error estimator remains effective.

Recalling (\ref{Def: error estimator}), (\ref{Def: e gl}) and (\ref{Def: e 0}), we propose the following local error estimator:
\begin{align*}
 \begin{aligned}
\varepsilon_T^2:=& \frac{1}{\nu^2}\| h_T^2
curl(\bar{\boldsymbol{u}}_h+\nu \Delta \bar{\boldsymbol{y}}_h)\|_{T}^2+\frac{1}{2}\sum_{E\in (\mathcal{E}_h\setminus \partial \Omega)\cap T}\|h_E^{1/2}[ \nabla \bar{\boldsymbol{y}}_h\cdot
  \boldsymbol{n}_E]\|_{E}^2\\
    &+\frac{1}{2\nu^2}\sum_{E\in (\mathcal{E}_h\setminus \partial \Omega)\cap T}\|h_E^{3/2}[(\bar{\boldsymbol{u}}_h+\nu  
    \Delta\bar{\boldsymbol{y}}_h)\cdot \boldsymbol{\tau}_E]\|_{E}^2+\|\nabla\cdot\bar{\boldsymbol{y}}_h\|_T^2\\
    &+\frac{1}{\nu^2}\| h_T^2
curl(\mathcal{R}_h\bar{\boldsymbol{y}}_h-\boldsymbol{y_d}+\nu \Delta\bar{\boldsymbol{z}}_h)\|_{T}^2+\frac{1}{2}\sum_{E\in (\mathcal{E}_h\setminus \partial \Omega)\cap T}\|h_E^{1/2}[ \nabla \bar{\boldsymbol{z}}_h\cdot
  \boldsymbol{n}_E]\|_{E}^2\\
    &+\frac{1}{2\nu^2}\sum_{E\in (\mathcal{E}_h\setminus \partial \Omega)\cap T}\|h_E^{3/2}[(\mathcal{R}_h\bar{\boldsymbol{y}}_h-\boldsymbol{y_d}+\nu  \Delta\bar{\boldsymbol{z}}_h))\cdot \boldsymbol{\tau}_E]\|_{E}^2+\|\nabla\cdot\bar{\boldsymbol{z}}_h\|_T^2\\
    &+\|\frac{1}{\alpha}\bar{\boldsymbol{z}}_h+\bar{\boldsymbol{u}}_h\|_T.
\end{aligned}
\end{align*}
It is evident that $\varepsilon_{pr}^2=\sum_{T\in\mathcal{T}_h}\varepsilon_T^2$. Utilizing the local error estimator $\varepsilon_T$, global error estimator $\varepsilon_{pr}$, and a D\"{o}fler parameter, we implement a standard adaptive refinement strategy, which encompasses iterative cycles structured as: ${\verb"SOLVE"}\rightarrow \verb"ESTIMATE"\rightarrow
\verb"MARK"\rightarrow \verb"REFINE"$. A more refined triangulation is obtained by subdividing the marked elements into four parts through regular refinement (bisecting all edges of the selected triangles), a process that can be executed using the MATLAB function REFINEMESH.

Subsequently, setting the D\"{o}fler parameter to $0.2$, we compare the convergence rates between the approximation error $E_h$ and the \textit{a posteriori} error estimator $\varepsilon_{pr}$.
As depicted in Tab.~\ref{tab:error ex2} and
Fig.~\ref{fig:convergence order ex2}, it is observed that
$E_h$, $\varepsilon_{pr}$, and their respective components exhibit convergence rates that are nearly optimal, approaching the value of
$1$ with respect to $d.o.f$. The two figures in Fig.~\ref{fig:Mesh and error ex2} illustrate the initial mesh with $d.o.f=252$ and the seventh refinement mesh with $d.o.f=5078$. The refinement process introduces additional elements near the origin as $d.o.f$ increases, thus diminishing the global error. The effectivity index
$\kappa_{pr}$, presented in the sixth column of Tab.~\ref{tab:error ex2}, 
converges to a constant value that is independent of $h$, signifying the reliability of the estimator.

\begin{figure}[htbp]
  \centering
  \includegraphics[width=0.8\textwidth]{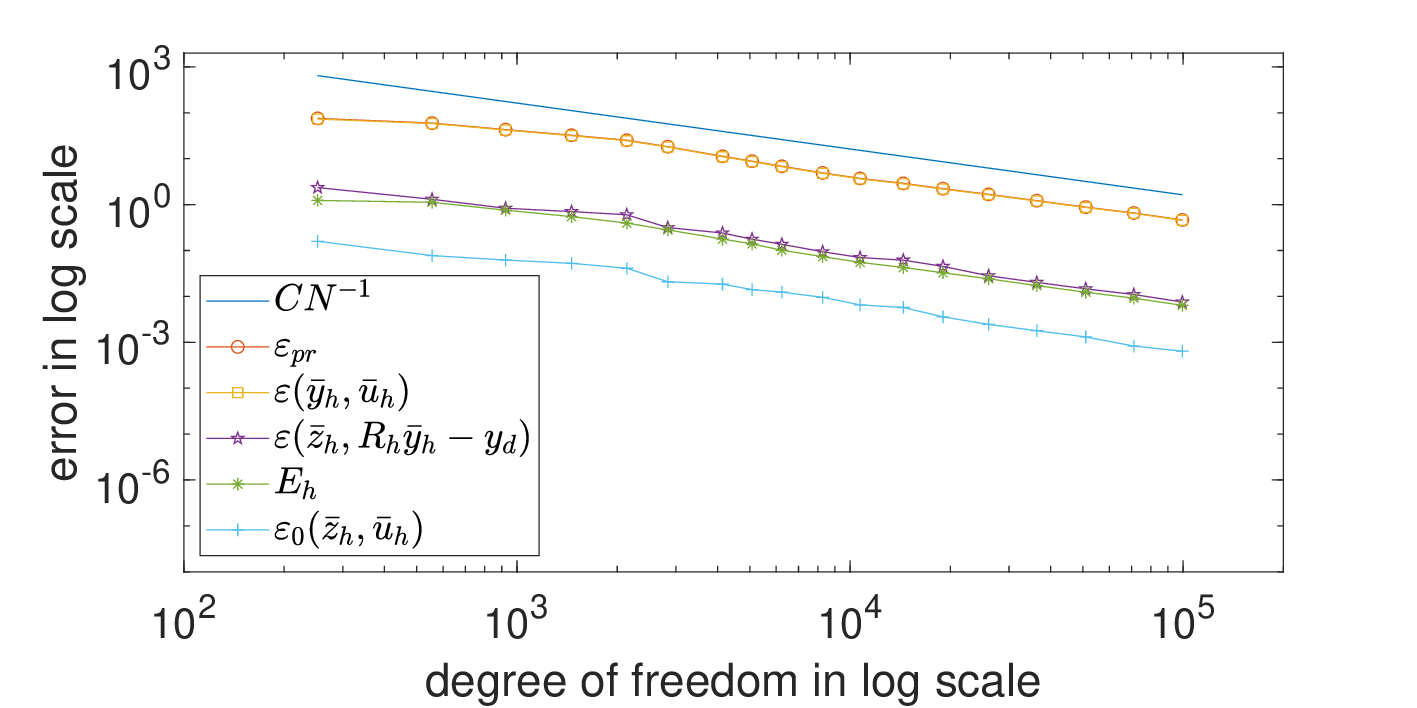}
\caption{The convergent rates for Example 2.}
 \label{fig:convergence order ex2}
\end{figure}

\begin{table}[htbp]
\centering \caption{The error of the discrete problem
and the \textit{a posteriori} error estimator in Example 2.} \label{tab:error ex2}
\begin{tabular*}{\hsize}{@{}@{\extracolsep{\fill}}cccccc@{}}
    \hline
   $d.o.f$          &  $E_h$ & order   &  $\varepsilon_{pr}$ & order & $\kappa_{pr}$\\
    \hline
      $252$   &  $1.232$E0      &   ---   &  $7.552$E1  &   ---          &64.27 \\

      $556$    &  $1.122$E0    &    $0.118$    &  $5.971$E1  &   $0.296$ &53.19 \\

      $924$    &  $7.548$E-1    &    $0.781$    &  $4.297$E1  &   $0.647$ & 56.93\\

     $1458$    &  $5.471$E-1    &    $0.705$    &  $3.243$E1  &   $0.617$ & 59.27\\

     $2136$    &  $3.960$E-1    &    $0.846$    &  $2.524$E1  &   $0.656$& 63.73\\

     $2836$    &  $2.803$E-1   &    $1.218$    &  $1.827$E1  &   $1.139$  &65.17\\

     $4140$    &  $1.744$E-1   &    $1.208$    &  $1.128$E1  &   $1.274$  &63.58\\

     $5078$    &  $1.392$E-1   &    $1.187$    &  $8.832$E0  &   $1.200$ &63.41\\

    $6240$    &  $1.012$E-1   &    $1.546$    &  $6.838$E0  &   $1.241$  &67.52\\

    $8272$    &  $7.350$E-2   &    $1.137$    &  $4.906$E0  &   $1.177$  &66.75\\

    $10716$    &  $5.517$E-2   &    $1.107$    &  $3.714$E0  &   $1.075$ &67.31\\

    $14448$    &  $4.257$E-2   &    $0.867$    &  $2.905$E0  &   $0.821$  &68.24\\

    $19006$   &  $3.272$E-2     &   $0.959$   &  $2.223$E0  &   $0.976$  &67.92 \\

    $26004$   &  $2.417$E-2     &   $0.962$   &  $1.671$E0  &   $0.905$  & 69.15 \\

    $36344$   &  $1.716$E-2     &   $1.027$   &  $1.216$E0  &   $0.953$  & 70.89 \\

    $50988$   &  $1.229$E-2     &   $0.983$   &  $8.801$E-1  &   $0.956$  & 71.55 \\

    $71090$   &  $9.134$E-3     &   $0.989$   &  $6.561$E-1  &   $0.883$  & 71.83 \\

    $99382$   &  $6.384$E-3     &   $1.069$   &  $4.625$E-1  &   $1.043$  & 72.45 \\  
    \hline
\end{tabular*}
\end{table}

\begin{figure}[htbp]
  \centering
  \includegraphics[width=0.45\textwidth]{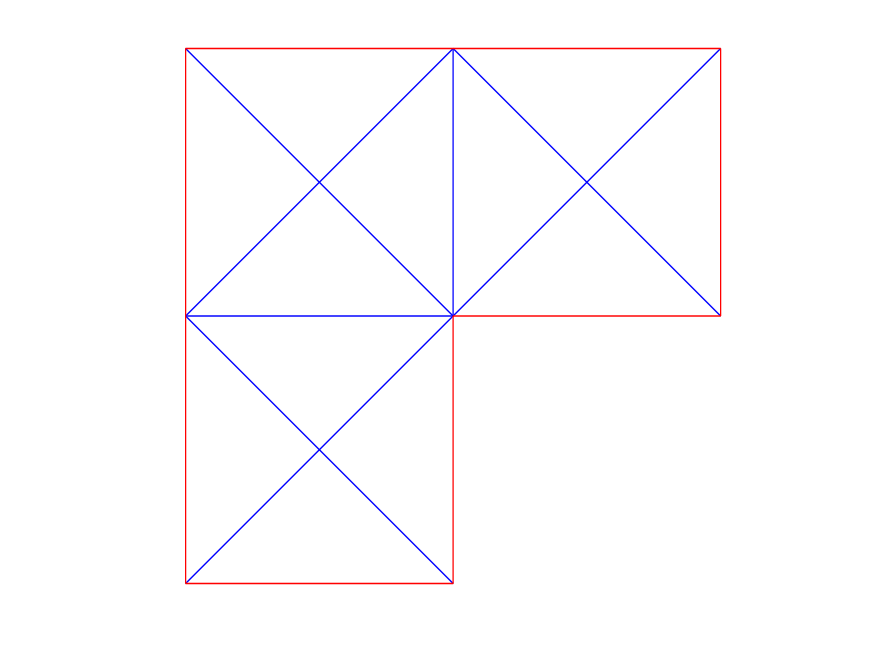}\quad
  \includegraphics[width=0.45\textwidth]{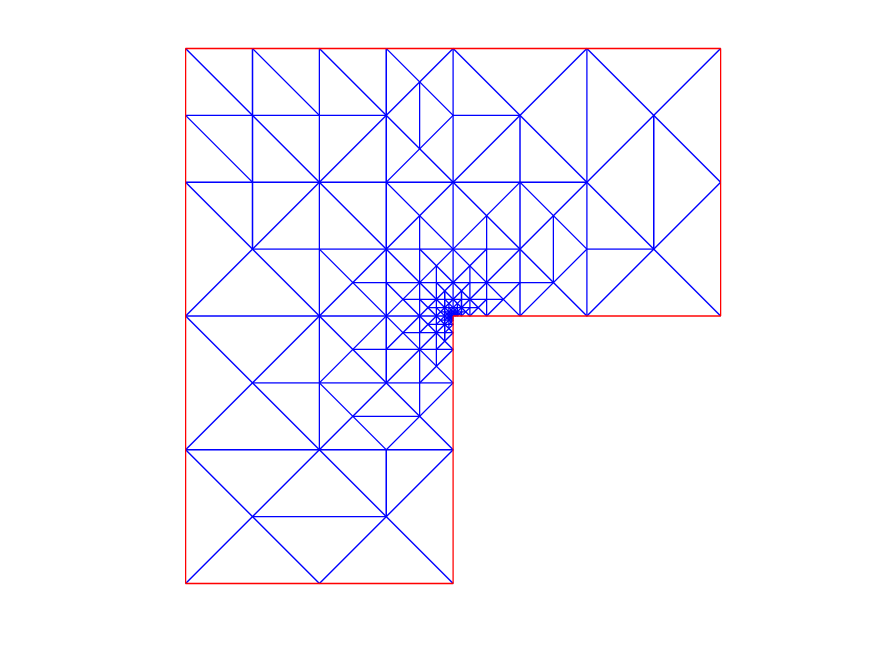}
\caption{The initial mesh (left) and the seventh refinement mesh (right) for Example 2}.
 \label{fig:Mesh and error ex2}
\end{figure}

\textbf{Example 3.} 
This example, adapted from \cite{correaUnifiedMixedFormulation2009}, is modified with an adjustable parameter $\gamma$ to compare the behavior of the effectivity indices of the new and classical a posteriori error estimators defined in (\ref{Def: e gl}) and (\ref{Def: error estimator classical}), 
respectively, under varying pressure magnitudes. 
Let $\Omega$ be the unit square. 
We fix $\nu = 1$ and $\alpha = 0.01$, while selecting $\gamma$ values from a geometric sequence to gradually increase the pressure magnitude. 
The adjoint state functions are chosen as $\bar{\boldsymbol{z}} = 0.01\,\bar{\boldsymbol{y}}$ and $\bar{r} = -0.01\,\bar{p}$, and the control function is given by $\bar{\boldsymbol{u}} = -\bar{\boldsymbol{z}}/\alpha$. 
As in Example~2, an extra source term arises on the right-hand side of the equation.

\begin{table}[htbp]
\centering \caption{The effectivity indices and computational cost for the new and classical \textit{a posteriori} error estimators with $d.o.f=37380$ in Example 3.} \label{tab:effective index ex3}
\begin{tabular*}{\hsize}{@{}@{\extracolsep{\fill}}c|cc|ccc|ccc@{}}
    \hline
   $\gamma$          &  $E_h$ & time(s)   &  $\varepsilon_{pr}$ & $\kappa_{pr}$ & time(s)
   & $\hat{\varepsilon}$ & $\kappa_c$ & time(s)\\
    \hline
      $1$   &  $2.925$E-4      &   $6.420$   &  $2.294$E-2  &   $78.424$          &$0.798$ 
      &  $1.464$E-1  &   $5.006$E2          &$0.239$\\

      $10$    &  $2.925$E-4    &    $6.031$    &  $2.294$E-2  &   $78.424$ &$0.762$ 
      &  $4.368$E-1  &   $1.493$E3 &$0.235$\\

      $10^2$    &  $2.925$E-4    &    $8.301$    &  $2.294$E-2  &   $78.424$ &$0.534$
      &  $4.148$E0  &   $1.417$E4 & $0.159$\\

     $10^3$    &  $2.925$E-4    &    $6.206$    &  $2.294$E-2  &   $78.425$ &$0.811$
     &  $4.147$E1  &   $1.417$E5 & $0.152$\\

     $10^4$    &  $2.925$E-4    &    $5.026$    &  $2.294$E-2  &   $78.510$&$0.831$
     &  $4.147$E2  &   $1.417$E6& $0.188$\\

     $10^5$    &  $2.925$E-4   &    $8.884$    &  $2.294$E-2  &   $81.404$  &$0.546$
     &  $4.147$E3  &   $1.417$E7  &$0.202$\\  
    \hline
\end{tabular*}
\end{table}

Tab.~\ref{tab:effective index ex3} reports the computed effectivity indices of the new and classical \textit{a posteriori} estimators for different values of~$\gamma$. 
As the pressure magnitude increases, the effectivity index of the new estimator remains essentially constant, confirming its pressure independence. 
In contrast, the effectivity index of the classical estimator increases significantly with~$\gamma$, indicating that it is strongly affected by the pressure. 
Fig.~\ref{fig:effective index ex3} illustrates the same trend graphically.

\begin{figure}[htbp]
  \centering
  \includegraphics[width=0.5\textwidth]{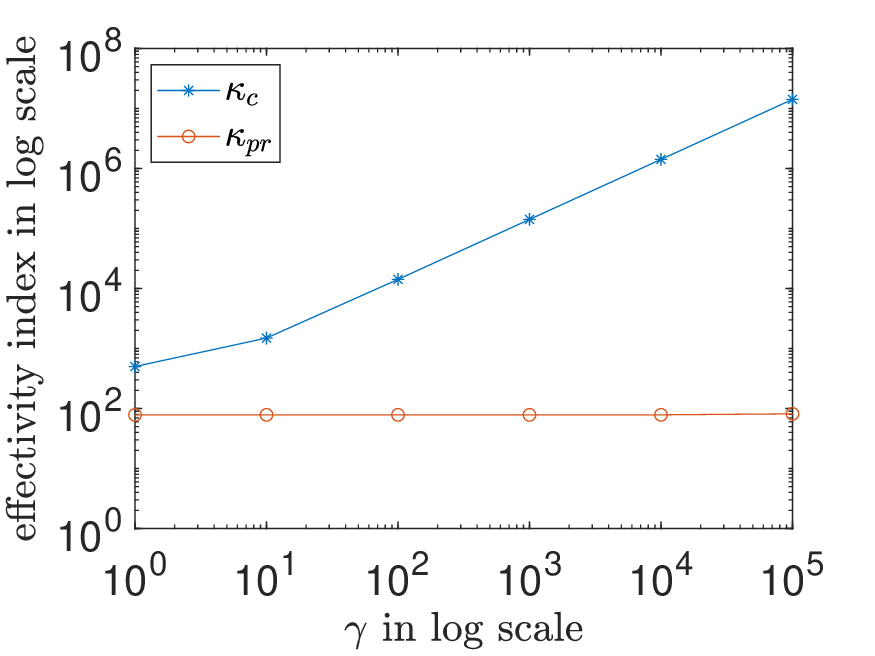}\quad
\caption{The effectivity indices for the new and classical \textit{a posteriori} error estimators with $d.o.f=37380$ in Example 3}.
 \label{fig:effective index ex3}
\end{figure}

This observation is consistent with Theorems~\ref{thm3},\ref{thm2}, and \ref{thm4}, and confirms that the new estimator is robust against pressure effects. 
Although the new estimator requires slightly more computational effort due to the inclusion of the divergence-free reconstruction operator, its computational cost remains acceptable compared with solving the discrete problem (shown in Tab.~\ref{tab:effective index ex3}). 
These results clearly demonstrate the superior performance of the proposed estimator in capturing the velocity error independent of the pressure.

\begin{remark}
Although the theoretical framework developed in this paper applies to both two- and three-dimensional cases, the numerical implementation in three dimensions is considerably more challenging. The main difficulty lies in constructing and applying the divergence-free reconstruction operator 
on tetrahedral meshes, which requires more complex local Raviart–Thomas spaces and significantly increases computational cost and memory usage. Moreover, evaluating the $curl$-based volume terms in the a posteriori error estimator (4.7) becomes more involved in three dimensions due to the vector-valued nature of the $curl$ operator.
\end{remark}

\noindent\textbf{Acknowledgments}
The work is supported by the National Key
 Research and Development Program of China (grant No. 2024YFA1012300, 2024YFA1012302), the National Natural Science Foundation of China (grant No. 12201310, 12301469), and the Natural Science Foundation of Jiangsu Province (grant No. BK20210540). 

\noindent\textbf{Data Availability} Data sharing not applicable to this article as no datasets were generated or analysed during the current study..

\section*{Declarations}
\noindent\textbf{Conflict of interest} The authors have no conflicts of interest to declare.

\bibliographystyle{unsrt}

\end{document}